\documentclass[11pt]{article}

\usepackage{amsmath,amsxtra,amssymb,latexsym,epsfig,amscd,amsthm,subfigure,fancybox}
\usepackage[mathscr]{eucal}
\usepackage{graphicx}
\usepackage{epsfig}
\usepackage{epstopdf}
\usepackage{cases}
\usepackage{multicol,xcolor}
\usepackage{enumerate}
\newcommand{\lf}{\lfloor}
\newcommand{\rf}{\rfloor}

\newtheorem{thm}{Theorem}[section]
\newtheorem {asp}{Assumption}[section]
\newtheorem{lm}{Lemma}[section]

\theoremstyle{definition}
\newtheorem*{conv*}{Convention}

\theoremstyle{remark}

\numberwithin{equation}{section}

\newcommand{\eps}{\varepsilon}

\newcommand{\M}{\mathcal{M}}
\newcommand{\K}{\mathcal{K}}
\newcommand{\F}{\mathcal{F}}

\newcommand{\E}{\mathbb{E}}

\newcommand{\N}{\mathbb{N}}
\newcommand{\PP}{\mathbb{P}}

\newcommand{\R}{\mathbb{R}}

\newcommand{\cd}{(\cdot)}

\numberwithin{equation}{section}
\newcommand{\1}{\boldsymbol{1}}

\newcommand{\bed}{\begin{displaymath}}
\newcommand{\eed}{\end{displaymath}}
\newcommand{\bea}{\bed\begin{array}{rl}}
\newcommand{\eea}{\end{array}\eed}
\newcommand{\ad}{&\!\!\!\disp}
\newcommand{\aad}{&\disp}
\newcommand{\barray}{\begin{array}{ll}}
\newcommand{\earray}{\end{array}}

\def\disp{\displaystyle}

\def\trace{\text{trace}}
\def\bar{\overline}
\def\a.s{\text{\;a.s.\;}}

\begin{document}
\title{Limit Cycles of Dynamic Systems under Random Perturbations
with Rapid Switching and Slow Diffusion: A Multi-Scale Approach}
\author{Dang Hai Nguyen,\thanks{Department of Mathematics, Wayne State University, Detroit, MI
48202, USA,
dangnh.maths@gmail.com.
Research of this author was supported in part by the National Science Foundation under Grant DMS-1207667.} \and
Nguyen Huu Du,\thanks{Department of Mathematics, Mechanics and
Informatics, Hanoi National University,
 334 Nguyen Trai, Thanh Xuan, Hanoi Vietnam, dunh@vnu.edu.vn. This research was
supported in part by
NAFOSTED
n$_0$ 101.02 - 2011.21.}
\and
 George Yin\thanks{Corresponding author: Department of Mathematics, Wayne State University, Detroit, MI
48202, USA, gyin@math.wayne.edu. Research of this author was supported in part by the National Science Foundation under Grant DMS-1207667.}}
\maketitle

\begin{abstract}
 This work is devoted to examining qualitative properties of dynamic systems, in particular, limit cycles of 
stochastic differential equations with both rapid switching and small diffusion.
The systems are featured by multi-scale formulation, highlighted
by the presence
 of two small parameters $\eps$ and $\delta$.
 Associated with the underlying systems, there are averaged or limit systems.
 Suppose that for each pair of the parameters, the solution of the corresponding equation has an invariant probability measure $\mu^{\eps,\delta}$,
and that the averaged equation has a limit cycle in which there is an averaged occupation measure $\mu^0$ for the averaged equation. Our main
effort is to prove that $\mu^{\eps,\delta}$  converges weakly to $\mu^0$ as $\eps \to 0$ and $\delta \to 0$ under suitable conditions.
Moreover, our
results are applied to a stochastic predator-prey model together with
numerical examples for demonstration.



\end{abstract}

\newpage

\section{Introduction}\label{sec:int}
It is
widely recognized that the traditional dynamical systems
 given by deterministic differential equations are
 often inadequate to model
  many natural phenomena because more often than not the systems are  affected by a variety of
  random perturbations.
  For this reason,
  various random noise processes
  have  been taken into consideration.
  Among them, diffusion processes are one of the most popular models.
 If the intensity of the diffusion is small,
   the diffusion process can be approximated by an ordinary differential
  equation.
Consider  a stochastic differential equation
 \begin{equation}\label{eq1.1}			
dx^\eps(t)=f(x^\eps(t))dt+\sqrt{\eps}\sigma(x^\eps(t))dW(t)
\end{equation}
for appropriate functions $f\cd$ and $\sigma\cd$,
where $W(\cdot)$ is a standard Brownian motion and $\eps$ is a small parameter.
It is intuitive that as $\eps\to 0$, \eqref{eq1.1} should have a limit that is
represented by a purely deterministic differential equation.
Building on
using averaging ideas,
 in \cite{WF},
  Fleming considered the associated
 asymptotic expansions of the expectation $\E_x\Phi(x^\eps(t))$ for a suitable
 function $\Phi\cd$
  under appropriate conditions.
 If the process $x^\eps(t)$ has a unique ergodic measure $\mu^\eps$ for each $\eps>0$ and  the origin of the corresponding deterministic equation
\begin{equation}\label{eq1.2}
dx=f(x)dt,
\end{equation}
is a globally asymptotic equilibrium point,
Holland \cite{CH1} established asymptotic expansions of the expectation of
 the underlying functionals with respect to the unique ergodic measure $\mu^\eps$.
 Furthermore, in  \cite{CH}, he considered the case that
 \eqref{eq1.2} has an asymptotically stable limit cycle and proved the weak convergence of the unique ergodic measure
 $\mu^\eps$ of $x^\eps(t)$ to the stationary distribution of \eqref{eq1.2} concentrated on the limit cycle.

In recent years, to enlarge the applicability, resurgent efforts have been devoted to modeling dynamic systems  in which
 continuous dynamics and discrete events coexist.
  Such  ``hybrid'' systems
 arise from
 traditional applications in engineering, operations research,  biological, and physical sciences
 as well as from emerging applications in wireless
communications,
internet traffic modeling, and financial engineering
among others; see \cite{YZ} and references therein.
As a result,
 switching diffusion processes have received much attention lately.
For such models, it is important to consider perturbed dynamic systems similar to that of the aforementioned paragraphs.
In the literature,
to take advantages of the time-scale separation,
Simon
and Ando \cite{SimonA} introduced the
so-called hierarchical decomposition and aggregation as well as nearly decomposable
models with economics applications; Sethi and
Zhang \cite{SethiZ94} initiated the study of near-optimal controls
for flexible manufacturing systems. A multi-scale approach with applications to computing and  Monte
Carlo simulation was treated in Liu \cite{Liu};
 applications of controlled dynamic systems can also be found in \cite{HYZ} and references therein.
 Recently,  effort has also been devoted to treat
 non-autonomous lattice systems with switching in Han and Kloeden \cite{HK16}.
 Related works have also been devoted to attraction, stability, and robustness
 of functional stochastic differential equations in Wu and Hu \cite{WH11} and to
 chemical reaction systems with two-time scales in Han and Najm \cite{HN14}.

In this paper, we consider dynamic systems represented by switching diffusions, where
the switching is rapidly varying whereas the diffusion is slowly changing.
To be more precise, let $(\Omega, \F, \{\F_t\}, \PP)$ be a filtered probability space satisfying the usual condition.
Consider the process
\begin{equation}\label{eq2.1}
dX^{\eps,\delta}(t)=f(X^{\eps,\delta}(t), \alpha^\eps(t))dt+\sqrt{\delta}\sigma(X^{\eps,\delta}(t), \alpha^\eps(t))dW(t)
, \ X^\eps(0)=x,\end{equation}
 where $W(t)$ is an $m$-dimensional
 standard Brownian motion,  $\alpha^\eps(t)$ is
 a Markov chain that is
 independent of $W(t)$
 and that has  a finite state space $\M=\{1,..., m_0\}$ and generator $Q/\eps=\big(q_{ij}/\eps\big)_{m_0\times m_0}$,
 $X^{\eps,\delta}$ is an $\R^d$-valued process, $f: \R^d\times\M\to\R^d, \sigma: \R^d\times\M\to\R^{d\times m},$ and
 $\eps>0$ and $\delta>0$ are two small parameters.
 Assume that $Q$ is irreducible. The irreducibility of $Q$ implies that the Markov chain  associated
 with $Q$ (denoted by $\tilde \alpha(t)$)
  is ergodic, with a unique stationary distribution $(\nu_1,...,\nu_{m_0})$.
Intuitively, when $\eps$ and $\delta$ are very small,
$\alpha^\eps(t)$ converges rapidly to its stationary distribution and the intensity of the diffusion is negligible. As a result,
on each finite interval of $t$, a solution of equation \eqref{eq2.1} can be approximated by
\begin{equation}\label{eq2.2}
dX(t)=\bar f(X(t))dt, \ X(0)=x,
\end{equation}
where $\bar f(x)=\sum_{i=1}^{m_0}f(x, i)\nu_i$.
However, if in lieu of a finite interval,  we consider the process in
the infinite interval $[0,+\infty)$.
The corresponding results
could be considerably different.
Suppose that equation \eqref{eq2.2} has a stable limit cycle, a natural question is whether invariant measures to
 the process represented by  solution of
 \eqref{eq2.1} converge weakly to the
  measure concentrated on the limit cycle. To address this question, we
substantially extend the results of \cite{CH} by considering the presence of both small diffusion and rapid switching.
Moreover,
because of the complexity due to the presence of the switching and coupling, new mathematical techniques need to be introduced to deal with the
current problem.
In addition, even we only consider the problem as that of \cite{CH},
it will be seen later that some of our assumptions are
weaker
than those used   in \cite{CH}.

The rest of the paper is organized as follows.
The main assumptions are given in Section \ref{sec:2}. Some auxiliary results are also presented. Section \ref{sec:3} takes
up the issue of estimation for the exit time of the solution from critical points that is needed to prove the main result. In Section \ref{sec:4}, we consider three cases. In each of the cases, we
address the aforementioned question under suitable conditions. Section
\ref{sec:5} deals with applications of our results to a general predator-prey model. It should be mentioned that the applications are not straightforward.
It is not easy to prove the existence and tightness of the family of invariant probability measures that are usually obtained using Lyapunov-type functions. However, the Lyapunov-type method does not work for our model. We  introduce certain tools to overcome the difficulties. Finally, we provide some numerical examples in Section \ref{sec:6} to illustrate our results.

\section{Assumptions and Preliminary Results}\label{sec:2}
 Throughout the paper, we denote by $A'$ the transpose of a matrix $A$,
   $|\cdot|$ the Euclidean norm of vectors in $\R^d$, and $\|A\|:=\sup\{|Ax|: x\in\R^d, |x|=1\}$ the operator norm of a matrix $A\in \R^{d\times d}$.
We use notations $a\wedge b=\min\{a, b\}$ and $a\vee b=\max\{a, b\}$.
Assuming in this paper that $\delta$ depends on $\eps$ ($\delta= \delta(\eps)$) and $\lim\limits_{\eps\to0}\delta (\eps)=0$,
we impose the following assumptions for equations \eqref{eq2.1} and \eqref{eq2.2}.
\begin{asp}\label{asp1} {\rm
\begin{enumerate}[(i)]
  \item
  For each $i\in \M$,
  $f(\cdot, i)$ and  $\sigma(\cdot, i)$ are locally Lipschitz continuous.
  \item There is an $a>0$ and a  twice continuously differentiable
   real-valued function $\Phi(x)\geq0$ such that $\lim\limits_{R\to\infty}\inf\{\Phi(x): |x|\geq R\}=\infty$ and that
  $\nabla \Phi'(x)f(x,i) \leq a(\Phi(x)+1)$,
   $\forall\, (x, i)\in\R^d\times\M$.
\item Equation \eqref{eq2.2} has a unique limit cycle denoted by $\Gamma.$
\item In any compact subset of $\R^d$, there is a finite number of critical
points of $\bar f(\cdot)$.
\item For each compact set $\K$ not containing any critical point of $\bar f$, and $\gamma>0$, there exists a $T=T_{\K,\gamma}>0$ such that $\inf\limits_{y\in\Gamma}|\bar X_x(t)- y|<\gamma$ for all $t\geq T$.	
\item There is an $\eps_0>0$ such that for all $0<\eps<\eps_0$, there exists a unique solution
to equation \eqref{eq2.1}.
The   process $(X(t), \alpha^{\eps,\delta}(t))$ has the strong Markov property and has an invariant measure $\mu^{\eps,\delta}$.
As a convention, given $A\subset\R^d$, we denote $\mu^{\eps,\delta}(A\times\M)$ by $\mu^{\eps,\delta}(A)$ for convenience.
The family $\{\mu^{\eps,\delta}: 0<\eps<\eps_0\}$ is tight in the sense that for any $\gamma>0$, there exists an $R>0$ such that $\mu^{\eps,\delta}\{B_R\}>1-\gamma$ for all $0<\eps<\eps_0$, where $B_R=\{x:|x|\leq R\}$.
\end{enumerate}
}\end{asp}

For simplicity, we
 have assumed that
 \eqref{eq2.1} has a unique solution; we have also assumed
  that the associated process is strong Markov.
  Sufficient conditions ensuring these can be provided (see \cite{MY, YZ}).
  Since these are already
known and are not our main concern here,
  we  use the conditions  as stated above.
Let $T_\Gamma$ be the period of the limit cycle $\Gamma$. For any $y\in\Gamma$, we can define a probability measure $\mu^0$ (not dependent on $y$) by
$$\mu^0(A)=\dfrac1{T_\Gamma}\int_0^{T_\Gamma} \1_{\{\bar X_y(s)\in A\}}ds,$$ where $\bar X_y(t)$ is the solution to equation \eqref{eq2.2} starting at $ \bar X(0)=y$ and $\1_{\{\cdot\}}$ is the indicator function.
Thus $\mu^0\cd$ is simply an averaged occupation measure.
Recall that $\delta = \delta_\eps$. We shall  address
asymptotic behavior
of $\mu^{\eps,\delta}$
for the following three cases:
\begin{equation}\label{eq:ep-dl}
\lim\limits_{\eps\to0}\dfrac{\delta}\eps=\left\{\begin{array}{ll}l\in(0,\infty), &\mbox{ case 1}\\0, &\mbox{ case 2}\\\infty, &\mbox{ case 3}.\end{array}\right.
\end{equation}
The multi-scale  modeling point was similar to \cite{HY14} in which large deviations were considered.
Here, we use such setting to study limit cycles of the associated dynamic systems.
We
impose  additional conditions corresponding to each of the cases in \eqref{eq:ep-dl}.
In case 2, since $\delta$ tends to $0$ much faster than $\eps$, if $\delta$ is sufficiently small, then intuitively,
the behavior of $X^{\eps,\delta}(t)$ should be similar to $\xi^\eps(t)$,
the solution to
\begin{equation}\label{eq2.3}
d\xi^\eps(t)=f\big(\xi^\eps(t),\alpha^\eps(t)\big)dt.
\end{equation}
Hence, if for each $i\in \M$, $f(x^*, i)=0$  at a critical point $x^*$ of $\bar f$, the Dirac distribution $\delta(x-x^*)$ is an invariant
measure for $\xi^\eps(t)$. Due to the above argument, the
 sequence of invariant probability measures $\mu^{\eps,\delta}$ (or one of its subsequences) may converge to $\delta(x-x^*)$. Consequently, we need to suppose that there is an $i^*\in\M$ such that $f(x^*, i^*)\ne0$ in order to obtain the convergence of $\mu^{\delta, \eps}$ to the measure $\mu^0$. Analogously, in case 3, we suppose that there exists an $i^*\in\M$ such that $\sigma(x^*, i^*)\ne0$.
 For case 1,  it is assumed that there is an $i^*\in\M$ satisfying either $\sigma(x^*, i^*)\ne0$ or $f(x^*, i^*)\ne0$.

The following formula is the well-known  exponential martingale inequality, which will be used several times in our proofs.
It asserts that
$$\PP\left\{\sup_{t\in[0, T]}\Big[\int_0^tg(s)dW(s)-\dfrac{a}{2}\int_0^tg^2(s)ds\Big]>b\right\}\leq e^{-ab},$$
if $g(t)$ is a real-valued $\F_t$-adapted process and $\int_0^Tg^2(t)dt<\infty$ almost surely (see \cite[Theorem 1.7.4]{XM}).

\begin{lm}\label{lm2.1}
For any $R, T, \gamma>0$, there is a $k_1=k_1(R, T, \gamma)>0$ such that for sufficiently small $\delta$,
$$\PP\{|X^{\eps,\delta}_{x, i}(t)-\xi^{\eps}_{x, i}(t)|\geq\gamma, \ \forall 0\leq t\leq T\}<\exp\Big(-\dfrac{k_1}{\delta}\Big), \ \forall |x|\leq R,$$
where $X^{\eps,\delta}_{x, i}(t)$ and $\xi^{\eps}_{x, i}(t)$ are solutions to
equations \eqref{eq2.1} and \eqref{eq2.3} with the same initial value $(x, i)$, respectively.
\end{lm}

\begin{proof}
By (i) and (ii) of Assumption \ref{asp1}, we can deduce the existence and boundedness of a unique solution to equation \eqref{eq2.3}
using
the
Lyapunov function method. Moreover, we can find an $R_T>0$ such that $|\xi^{\eps}_{x, i}(t)|<R_T-\gamma, \ \forall\, t\in[0, T]$ a.s. $\,\forall\,|x|\leq R$.
Let $h(x)$ be a twice differentiable function with compact support and $h(x)=1$ if $|x|\leq R_T$ and $h(x)=0$
if $|x|\geq R_T+1$.
Put $f_h(x, i)=h(x)f(x, i)$, $\sigma_h(x, i)=h(x)\sigma(x, i)$.
Let $Y^{\eps,\delta}_{x, i}(t)$ be the solution starting at $(x, i)$ to
\begin{equation}
d Y(t)= f_h(Y(t),\alpha^\eps(t))dt+\sqrt{\delta}\sigma_h(Y(t), \alpha^\eps(t)dW(t)
\end{equation}
Note that $Y^{\eps,\delta}_{x, i}(t)=X^{\eps,\delta}_{x, i}(t)$ up to the time $\zeta=\inf\{t>0: |X^{\eps,\delta}_{x, i}(t)|>R_T\}$.
Meanwhile $\xi^{\eps}_{x, i}(t)$,
the solution to \eqref{eq2.3},  is also the solution to
$$ d Z(t)= f_h(Z(t),\alpha^\eps(t))dt$$
for $|x|\leq R$ and $0\leq t\leq T$.
We have from the generalized It\^o's formula that
\begin{equation}
\begin{aligned}
|Y&^{\eps,\delta}_{x, i}(t)-\xi^{\eps}_{x, i}(t)|^2\\
\leq&2\int_0^{t}|Y^{\eps,\delta}_{x, i}(s)-\xi^{\eps}_{x, i}(s)|'|f_h(Y^{\eps,\delta}_{x, i}(s),\alpha^\eps(s))-f_h(\xi^{\eps}_{x, i}(s),\alpha^\eps(s))|ds\\
&+\int_0^{t}\delta\trace\big((\sigma_h\sigma_h')(Y^{\eps,\delta}_{x, i}(s),\alpha^\eps(s))\big)ds\\
&+2\sqrt{\delta}\left|\int_0^{t}\big(Y^{\eps,\delta}_{x, i}(s)-\xi^{\eps}_{x, i}(s)\big)'\sigma_h(Y^{\eps,\delta}_{x, i}(s),\alpha^\eps(s)\big)dW(s)\right|.
\end{aligned}
\end{equation}
By the exponential martingale inequality, for any $\delta<k_1$,
\begin{align*}
\PP(A)\geq 1-2\exp\Big(-\dfrac{2k_1}\delta\Big)\geq1-\exp\Big(-\frac{k_1}\delta\Big),
\end{align*}
where
\begin{align*}
A=\bigg\{\omega: &\Big|\int_0^{t}\sqrt{\delta}\big(Y^{\eps,\delta}_{x, i}(s)-\xi^{\eps}_{x, i}(s)\big)'\sigma_h(Y^{\eps,\delta}_{x, i}(s),\alpha^\eps(s)\big)dW(s)\Big|\\
&\qquad-\dfrac1\delta\int_0^{t}\delta\big|Y^{\eps,\delta}_{x, i}(s)-\xi^{\eps}_{x, i}(s)\big|^2\|\sigma_h\sigma_h'(Y^{\eps,\delta}_{x, i}(s),\alpha^\eps(s)\big)\|ds\leq k_1\, \text{ for all }\, t\in[0,T]\bigg\}.
\end{align*}
Since $f_h$ is Lipschitz and $\sigma_h$ is bounded, we have for $\omega\in A$ and some sufficiently large $M_1$,
\begin{equation}
\begin{aligned}
|Y^{\eps,\delta}_{x, i}&(t)-\xi^{\eps}_{x, i}(t)|^2\\
\leq & 2 \int_0^{t}|Y^{\eps,\delta}_{x, i}(s)-\xi^{\eps}_{x, i}(s)|'|f_h(Y^{\eps,\delta}_{x, i}(s),\alpha^\eps(s))-f_h(\xi^{\eps}_{x, i}(s),\alpha^\eps(s))|ds\\
&\ +2\int_0^{t}\big|Y^{\eps,\delta}_{x, i}(s)-\xi^{\eps}_{x, i}(s)\big|^2\|\sigma_h\sigma_h'(Y^{\eps,\delta}_{x, i}(s),\alpha^\eps(s)\big)\|ds\\
&\ +\int_0^{t}\delta\trace\big((\sigma_h\sigma_h')(Y^{\eps,\delta}_{x, i}(s),\alpha^\eps(s))\big)ds+2\int_0^tk_1ds\\
\leq & M_1\int_0^t|Y^{\eps,\delta}_{x, i}(t)-\xi^{\eps}_{x, i}(t)|^2ds+(2k_1+M_1\delta) t
.\end{aligned}
\end{equation}
 For each $t\in[0, T]$,
  Gronwall's inequality leads to
$$|Y^{\eps,\delta}_{x, i}(t)-\xi^{\eps}_{x, i}(t)|^2\leq (2k_1+M_1\delta)T\exp(M_1T)<\gamma^2$$ for $0<\delta<k_1$ sufficiently small.
It also follows from this inequality that for $\omega\in A$ and $0<\delta<k_1$ sufficient small, we have $\zeta>T$ which implies $$|X^{\eps,\delta}_{x, i}(t)-\xi^{\eps}_{x, i}(t)|^2=|Y^{\eps,\delta}_{x, i}(t)-\xi^{\eps}_{x, i}(t)|^2<\gamma^2,$$ $\forall 0\leq t\leq T.$
\end{proof}

\begin{lm}\label{lm2.1b}
For each $x$ and $\gamma$, we can find $k_{\gamma,x}=k_{\gamma,x}(T)>0$ such that
$$\PP\left\{|\xi^{\eps}_{x, i}(t)-\bar X_x(t)|\geq\gamma , \ \forall 0\leq t\leq T\right\}\leq\exp\Big(-\frac{k_{\gamma,x}}\eps\Big),$$
where $\bar X_x(t)$ is the solution to equation \eqref{eq2.2} with the initial value $x$.
\end{lm}

\begin{proof}
It follows immediately from the large deviation principle shown in \cite{HYZ} with a note that the existence and boundedness of a unique solution to equation \eqref{eq2.3} follows from (i) and (ii) of Assumption \ref{asp1}.
\end{proof}

Combining Lemmas \ref{lm2.1} and \ref{lm2.1b}, we obtain the following lemma.

\begin{lm}\label{lm2.2}
For any $R$, $T$, $\gamma>0$, there is a $k=k(R,\gamma, T)>0$ such that
$$\PP\left\{|X^{\eps,\delta}_{x,i}(t)-\bar X_x(t)|\geq\gamma, \ \forall  0\leq t\leq T\right\}<\exp\Big(-\frac{k}{\eps+\delta}\Big), \ \forall  |x|\leq R.$$
\end{lm}

\begin{proof}
By virtue of Lemma \ref{lm2.1b}, for each $x$ and $\gamma$, we have
$$\PP\left\{|\xi^{\eps}_{x, i}(t)-\bar X_x(t)|\geq\dfrac\gamma6, \ \forall  0\leq t\leq T\right\}\leq\exp\Big(-\frac{k_{\gamma/6,x}}\eps\Big).$$
Using the Lyapunov method to \eqref{eq2.3} in view of $(ii)$ of Assumption \ref{asp1},  there is an $H_{R, T}>0$ such that $|\xi^{\eps}_{x, i}(t)|\leq H_{R, T}$ and $|\bar X_x(t)|\leq H_{R, T}$ for all $|x|\leq R$ and $0\leq t\leq T$. Since $f(\cdot, i)$ is locally Lipschitz for all $i\in\M$, there is a constant $M_2>0$ such that $|f(u, i)-f(v, i)|\leq M_2|u-v|$ for all $|u|\vee|v|\leq H_{R, T}$ and $i\in\M$. Using the Gronwall
inequality, we have for $|x|\vee|y|\leq R$ and $i\in\M$ that
\bea \ad |\xi^{\eps}_{x, i}(t)-\xi^{\eps}_{y,i}(t)|\leq|x-y|\exp(M_2T)\mbox{ and}\\
\ad
|\bar X_x(t)-\bar X_y(t)|\leq|x-y|\exp(M_2T), \ \forall  t\in[0, T].\eea
Let $\lambda=\dfrac{\gamma}6\exp(-M_2T)$. It is easy to see that for $|x-y|<\lambda$,
$$\PP\left\{|\xi^{\eps}_{y,i}(t)-\bar X_y(t)|\geq\dfrac\gamma2, \ \forall 0\leq t\leq T\right\}\leq\exp\Big(-\frac{k_{\gamma/6,x}}\eps\Big).$$
By the compactness of $\{x: |x|\leq R\}$, for $\gamma>0$, there is a $k_2=k_2(R, T,\gamma)>0$ such that for $|x|\leq R$,
$$\PP\left\{|\xi^{\eps}_{x, i}(t)-\bar X_x(t)|\geq\dfrac\gamma2, \ \forall  0\leq t\leq T\right\}\leq \exp\Big(-\frac{k_2}\eps\Big).$$
Combining with Lemma \ref{lm2.1}, we have
\bea \disp \PP\left\{|X^{\eps,\delta}_{x,i}(t)-\bar X_x(t)|\geq\gamma, \ \forall  0\leq t\leq T\right\}\ad <\exp\Big(-\frac{k_1(R, T,\gamma/2)}\delta\Big)+\exp\Big(-\frac{k_2}\eps\Big)\\
\ad <\exp\Big(-\frac{k}{\eps+\delta}\Big)\eea
for a suitable $k$ and sufficiently small $\eps$ and $\delta$.
\end{proof}

\section{Probability Estimate for the First Exit Time}\label{sec:3}
Let $x^*$ be a critical point of $\bar f(\cdot)$. In this section, we  consider a sufficiently small neighborhood $N$ of $x^*$ and estimate the first time $X_{x,i}^{\eps, \delta}(\cdot)$ exits from $N$, say $\tau_{x,i}^{\eps, \delta}$.
Since we mainly consider the time that the process exits a small neighborhood, by arguments similar to those in the proof of Lemma \ref{lm2.1}, we can suppose that
$f$ and $\sigma$ are bounded.
Let $x^*=0$ for convenience.
Since we will assume later that there are $i^*\in \M$ and $\beta\in\R^d, |\beta|=1$ such that $\beta'f(0, i^*)>0$ or $\beta'\sigma(0, i^*)>0$, we choose $c_1>0$ sufficient small such that $S=\{y: |y|< c_1\}$ contains only one critical point 0 and that for all $x\in \bar S=\{y: |y|\leq c_1\}$ we have
$\beta'f(x, i^*)> 0$ or $\beta'\sigma(x, i^*)>0$ depending on which function we impose the condition on.
In view of (v) of Assumption \ref{asp1}, there is a $T_{c_1}>0$ such that $\bar X_x(t)\notin S$ for all $t>T_{c_1}$ and $x\in\partial S:=\{y: |y|=c_1\}$. It follows from the continuous dependence of solutions on initial
data that there is $0<c_2<c_1$ such that
$\inf\{|\bar X_x(t)|: x\in\partial S, t\in[0, T_{c_1}]\}>c_2.$
Consequently, $\inf\{\bar X_x(t): x\notin S, t\geq0\}> c_2$.
Similarly, we can find $0<c_3<\dfrac{c_2}2$ such that
$\inf\{\bar X_x(t): |x|\geq c_2, t\geq0\}\geq 2c_3$.
Put $N=\{y: |y|<c_3\}$, $G =\{y: |y|<c_2\}$, and let $\gamma\in\big(0,\min\{c_1-c_2, c_2-2c_3, c_3\}\big)$.
Then, the following assertions are true:
\begin{itemize}
  \item $N\subset G\subset S.$
  \item For any $x\notin S$, $\bar X_x(t)\notin G $ for all $t\geq0$.
  \item For any $x\notin G $, $y\in N$, we have $|\bar X_x(t)-y|>\gamma, \ \forall \,t\geq0$.
\end{itemize}

\begin{lm}\label{lm3.1}
Suppose that there is an $\ell>0$ such that $\PP\{\tau_{x, i}^{\eps, \delta}<\ell\}\geq a^{\eps,\delta}>0, \ \forall  (x, i)\in N\times\M$, where $\lim\limits_{\eps\to0}a^{\eps,\delta}=0$.
Then $\PP\left\{\tau_{x, i}^{\eps, \delta}<\dfrac{\ell}{a^{\eps,\delta}}\right\}>1/2$ for sufficiently small $\eps$.
\end{lm}

\begin{proof}
Let $n^{\eps,\delta}\in\N$ such that $n^{\eps,\delta}-1<\dfrac{1}{a^{\eps,\delta}}\leq n^{\eps,\delta}.$
We consider events $A_k=\{X_x^{\eps, \delta}(t)\in N, \ \forall (k-1)\ell <t\leq k\ell \}$.
We have $\PP(A_1)\leq1-a^{\eps,\delta}.$
By the Markov property,
\bea \disp\PP(A_k|A_1,...,A_{k-1})\ad =\int_{N}\PP\left\{\tau_y^{\eps, \delta}\leq \ell \right\}\PP\Big\{X_x^{\eps, \delta}((k-1)\ell )\in dy\Big|A_1,...,A_{k-1}\Big\}\\
\ad\leq1-a^{\eps,\delta}.\eea
As a result,
$$\PP(A_1A_2\cdots A_n)\leq (1-a^{\eps,\delta})^{n^{\eps,\delta}}$$
Since $\lim\limits_{\eps\to0}a^{\eps,\delta}=0$, we deduce that $\lim\limits_{\eps\to0}(1-a^{\eps,\delta})^{n^{\eps,\delta}}=e^{-1}$, which means that $(1-a^{\eps,\delta})^{n^{\eps,\delta}}<1/2$ for sufficiently small $\eps$.
\end{proof}

\begin{lm}\label{lm3.2}
If  there is an $i^*$ such that $f(0, i^*)\ne 0$, then for any $\Delta>0$, we can find $H^\Delta_1>0$,and $ \eps_{1}(\Delta)$ such that for $\eps<\eps_{1}(\Delta)$,
$$\PP\left\{\tau^{\eps,\delta}_{x, i}\leq H^\Delta_1\exp\Big(\dfrac{\Delta}{\eps}\Big)\right\}>1/2\,\forall\,(x,i)\in N\times\M.$$
\end{lm}

\begin{proof}
Let $\beta\in\R^d$, $|\beta|=1$ such that $\beta'f(0, i^*)>0$. Suppose that $\beta'f(x, i^*)>a_1>0$ for all $x\in S$.
We
 need only
work with a  small $\Delta$, so we can suppose  $h_\Delta:=\dfrac{a_1\Delta}{4|q_{i^*i^*}|}<c_3$.
 First, we consider the case $\alpha^\eps(0)=i^*$. Because of the independence
  of $\alpha^\eps(\cdot)$ and $W(\cdot)$, given that $\alpha^\eps(t)=i^*, \ \forall  t\in\Big[0, \frac{\Delta}{|q_{i^*i^*}|}\Big]$, $X^{\eps, \delta}_{x, i^*}(\cdot)$ has the same distribution on $\Big[0, \frac{\Delta}{|q_{i^*i^*}|}\Big]$ as the solution $Z^{\delta}_x$ of the equation
$$dZ(t)=f(Z(t), i^*)dt+\sqrt{\delta}\sigma(Z(t), i^*)dW(t).$$
Put $$\rho^{\eps,\delta}_x=\dfrac{\Delta}{|q_{i^*i^*}|}\wedge \inf\{t>0: |Z_x^\delta(t))|\geq h_\Delta\}.$$
If $|x|\geq h_\Delta$ then $\rho^{\eps,\delta}_x=0$. Hence, we just consider the case $|x|< h_\Delta$.
We have
$$\beta'Z_x^\delta(\rho_x^{\eps,\delta})=\beta'x+\int_0^{\rho^{\eps,\delta}_x}\beta'f(Z_x^\delta(s), i^*)ds+\int_0^{\rho^{\eps,\delta}_x}\sqrt{\delta}\beta'\sigma(Z_x^\delta(s), i^*)dW(s).$$
By the exponential martingale inequality,
\bea \ad
\PP\Big\{-\int_0^{\rho^{\eps,\delta}_x}\sqrt{\delta}\beta'\sigma(Z_x^\delta(s), i^*)dW(s)
\\ \aad \quad
-\dfrac1{\sqrt{\delta}}\int_0^{\rho^{\eps,\delta}_x}\delta\beta'\sigma(Z_x^\delta(s), i^*)\sigma(Z_x^\delta(s), i^*)'\beta ds>\sqrt{\delta}\Big\}
\leq e^{-2}<1/2.\eea
Thus
\bea \ad \PP\Big\{\beta'Z_x^\delta(\rho_x^{\eps,\delta})>\beta'x+\int_0^{\rho^{\eps,\delta}_x}\beta'f(Z_x^\delta(s), i^*)ds\\
\aad \quad -\dfrac1{\sqrt{\delta}}\int_0^{\rho^{\eps,\delta}_x}\beta'\delta\sigma(Z_x^\delta(s), i^*)'\sigma(Z_x^\delta(s), i^*)\beta ds-\sqrt{\delta}\Big\}
> 1/2.\eea
Consequently, there is some positive constant $M_3$ such that
\begin{equation}\label{eq3.1}\barray\ad
\PP\Big\{|Z_x^\delta({\rho^{\eps,\delta}_x})|\geq\beta'Z_x^\delta({\rho^{\eps,\delta}_x})>-|\beta||x|+ \int_0^{\rho^{\eps,\delta}_x}a_1ds-M_3\sqrt{\delta} \\
\aad \qquad \ge -h_\Delta+a_1{\rho^{\eps,\delta}_x}-M_3\sqrt{\delta}\Big\}> 1/2.
\earray\end{equation}
Let $\delta$ be so small that $M_3\sqrt{\delta}<\dfrac{a_1}2\dfrac{\Delta}{|q_{i^*i^*}|}$.
 If $\rho^{\eps,\delta}_x=\dfrac{\Delta}{|q_{i^*i^*}|}$,  we have
$$|Z_x^\delta({\rho^{\eps,\delta}_x})|\leq h_\Delta< -h_\Delta+a_1{\rho^{\eps,\delta}_x}-M_3\sqrt{\delta}.$$
We
claim from \eqref{eq3.1} and the above fact that for sufficiently small $\delta$,
$$\PP\Big\{\inf\{t>0: |Z_x^\delta(t)|>h_\Delta\}< \dfrac{\Delta}{|q_{i^*i^*}|}\Big\}>1/2\,\forall x: |x|\leq h_\Delta$$
which implies
$$\PP\bigg\{\eta^{\eps, \delta}_{x, i^*}\leq \dfrac{\Delta}{q_{i^*i^*}}\bigg|\alpha(t)=i^*, \ \forall  t\in\Big[0,\dfrac{\Delta}{|q_{i^*i^*}|}\Big]\bigg\}>1/2\,\forall x: |x|\leq h_\Delta,$$
where $\eta^{\eps, \delta}_{x, i}=\inf\Big\{t: \big|X^{\eps,\delta}_{x, i}(t)\big|>h_\Delta\Big\}.$
Consequently, $$\PP\Big\{\eta^{\eps, \delta}_{x, i^*}\leq \dfrac{\Delta}{|q_{i^*i^*}|}\Big\}> \dfrac12\PP\Big\{\alpha^\eps(t)=i^*, \ \forall  t\in\Big[0, \dfrac{\Delta}{|q_{i^*i^*}|}\Big]\Big\}=\dfrac12\exp(-\dfrac{\Delta}{\eps}).$$
Since $\alpha^\eps(t)$ is ergodic, for any sufficiently small $\eps$, $$\PP\{\alpha^\eps(t)=i^* \mbox{ for some } t\in [0, 1]|\alpha^\eps(0)=i\}>{1 \over 2}, \ \forall i\in\M.$$ By the strong Markov property,
\begin{equation}
\PP\left\{\eta^{\eps, \delta}_{x, i}<1+\dfrac{\Delta}{|q_{i^*i^*}|}\right\}\geq {1\over 4}\exp\left(-\dfrac{\Delta}{\eps}\right), \ \forall \,(x,i)\in N\times\M.
\end{equation}
It is easy to see that if $|y|\geq h_\Delta,$ by  (v) of Assumption \ref{asp1}, there exists a $T_{h_\Delta}>0$ such that $\big|\bar X_y\big(T_{h_\Delta}\big)\big|\notin G .$
In view of Lemma \ref{lm2.2}, for $\eps$
sufficiently small,
\begin{align*}
\PP\Big\{&\Big|X^{\eps,\delta}_{y, i}\big(T_{h_\Delta}\big)-\bar X_y\big(T_{h_\Delta}\big)\Big|<\gamma\Big\}> {1 \over 2}, \ \forall  y\in S, i\in \M.
\end{align*}
Since $\bar X_y\big(T_{h_\Delta}\big)\notin G $, it follows from the construction of $G $ and $N$,  when $\eps$ is small, $\PP\Big\{X^{\eps,\delta}_{y, i}\big(T_{h_\Delta}\big)\notin N\Big\}>1/2$ provided that $ h_\Delta\leq|y|\leq c_2.$ As a result,
\begin{equation}
\PP\left\{\tau^{\eps, \delta}_{y, i}<1+\dfrac{\Delta}{|q_{i^*i^*}|}+T_{h_\Delta}\right\}> {1 \over 8}\exp\left(-\dfrac{\Delta}{\eps}\right), \ \forall  y\in N
\end{equation}

Using Lemma \ref{lm3.1},
for sufficiently small $\eps$,
$$\PP\left\{\tau^{\eps, \delta}_{x, i}<H^\Delta_1\exp\Big(\dfrac{\Delta}{\eps}\Big)\right\}\geq {1\over 2}\,\forall \,(x,i)\in N\times\M,$$
where $H^\Delta_1:=8\Big(1+\dfrac{\Delta}{|q_{i^*i^*}|}+T_{h_\Delta}\Big).$
\end{proof}

\begin{lm}\label{lm3.3}
Suppose that $\lim\limits_{\eps\to0}{\delta\over\eps}=l>0$ and $f(0, i)=0$ for all $i$, and there is an $i^*$ satisfying $\sigma(0, i^*)\ne 0.$ Then, for every $\Delta>0$, there are $H^\Delta_2>0$ and $\eps_{2}(\Delta)>0$ such that for $\eps<\eps_{2}(\Delta)$,
$$\PP\left\{\tau^{\eps, \delta}_{x,i}<H^\Delta_2\exp\Big(\dfrac{\Delta}\delta\Big)\right\}>1/2, \ \forall  \,(x,i)\in N\times\M.$$
\end{lm}

\begin{proof}
Without loss of generality, we consider only the case  $\lim\limits_{\eps\to0}{\delta\over\eps}=1$. $\sigma(0, i^*)\ne 0$ implies that there is $\beta\in\R^d$, $|\beta|=1$ such that $\beta'\sigma(0, i^*)>0$.
Suppose that $0<a_2<\beta'(\sigma\sigma')(y, i^*)\beta$, $\forall y\in S$.
Define
$$\zeta_t:=\inf\Big\{u>0:\int_0^u\beta'(\sigma\sigma')\big(X^{\eps, \delta}_{x, i}(s\wedge\tau^{\eps,\delta}_{x, i}),\alpha^\eps(s\wedge\tau^{\eps,\delta}_{x, i})\big)\beta ds\geq t\Big\}.$$
Then
 for all $t$, $\zeta_t<\infty$ a.s. and $$M(t)=\int_0^{\zeta_t}\beta'\sigma\big(X^{\eps, \delta}_{x, i}(s\wedge\tau^{\eps,\delta}_{x, i}),\alpha^\eps(s\wedge\tau^{\eps,\delta}_{x, i})\big)dW(s)$$ is a Brownian motion. This claim stems from the fact that $M(t)$ is a continuous martingale with quadratic variation $t$. Since $M(1)$ is standard normally distributed, for $\Delta<1$ and $\delta$ is sufficiently small, we have the estimate
$$\PP\{\sqrt{\delta}M(1)>\Delta\}\geq\dfrac12\exp\Big(-\dfrac\Delta{\delta}\Big).$$
Using the large deviation principle (see \cite{HYZ}), we
claim that there is $a_3=a_3(T)>0$ such that
$$\PP\left\{\dfrac{1}{T}\int_0^{T}\1_{\{\alpha^\eps(s)=i^*\}}ds>\dfrac{\nu_{i^*}}2\right\}\geq1-\exp\left(-\dfrac{a_3}{\eps}\right).$$
Let $T,
\Delta $ be such that $\dfrac{a_2\nu_{i^*}T}{2}>1$ and $\Delta<a_3$. Since $\lim\limits_{\eps\to0}{\delta\over\eps}=1$, for sufficiently small $\eps$
we have
\bea \ad \PP\left\{\int_0^{T}\beta'(\sigma\sigma')(X^{\eps,\delta}_{x, i}(s\wedge\tau^{\eps,\delta}_{x, i}))\beta ds
 \geq \dfrac{a_2\nu_{i^*}T}{2} \right\}\geq1-\exp\left(-\dfrac{a_3}{\eps}\right)\geq 1-\dfrac14\exp\left(-\dfrac{\Delta}{\delta}\right),\eea
so
$$\PP\{\zeta_{1}\leq T\}\geq1-\dfrac14\exp\left(-\dfrac{\Delta}{\delta}\right).$$
Since $f(0, i)=0, \ \forall  i\in\M$, we can find a $\theta=\theta(\Delta)>0$ such that $\{y:|y|\leq\theta\}\subset N$ and that $2\theta+T\times\sup\limits_{|x|<\theta, i\in\M}\{|f(x, i)|\}<\Delta.$
Let $\eta^{\theta}_{x,i}=\inf\{t>0: |X^{\eps,\delta}_{x,i}(t)|\geq\theta\}$.
Consider only the case $|x|<\theta$. Note that if $\sqrt{\delta}M(1)>\Delta$ and $\zeta_1\leq T$, we must have
$\eta^\theta_{x, i}<\zeta_1\leq T.$
Indeed, if the three events $\{\sqrt{\delta}M(1)>\Delta\}$, $\{\zeta_1\leq T\}$, and $\{\eta^\theta_{x, i}\geq\zeta_1\}$ happen simultaneously, it results in a contradiction that
\begin{align*}\Delta
& <\sqrt{\delta}M(1)=\sqrt{\delta}\int_0^{\zeta_1}\beta'\sigma(X^{\eps,\delta}_{x, i}(s), \alpha^\eps(s))dW(s)\\
& \leq |\beta'X^{\eps, \delta}_{x, i}(\zeta_1)|+|\beta'x|+\Big|\int_0^{\zeta_1}\beta'f(X^{\eps, \delta}_{x, i}(s),\alpha^\eps(s))ds\Big|\\
& \leq2\theta+\int_0^{\zeta_1}|f(X^{\eps, \delta}_{x, i}(s),\alpha^\eps(s))|ds< \Delta.
\end{align*}
Thus
\begin{equation}\label{e3.5}
\PP\{\eta^{\theta}_{x,i}<T\}\geq\PP\{\sqrt{\delta}M(1)>\Delta, \zeta_1\leq T\} \geq\dfrac14\exp(-\dfrac\Delta\delta)\,\forall\, |x|<\theta, i\in\M.
\end{equation}
By (v) of Assumption \ref{asp1}, there is a $T_\theta$ such that $|\bar X_y(T_\theta)|\geq c_2$ for all $y$ satisfying $|y|\geq\theta$.
In view of Lemma \ref{lm2.2}, when $\eps$ is sufficiently small, for all $\theta\leq |y|\leq c_2$,
\begin{equation}\label{e3.6}
\PP\{\tau^{\eps,\delta}_{y, i}<T_\theta\}\geq\PP\big\{\big|X^{\eps,\delta}_{y, i}(T_\theta)-\bar X_y(T_\theta)\big|<\gamma\big\}>1/2.
\end{equation}
Applying the strong Markov property,
\eqref{e3.5} and \eqref{e3.6} yield
$$\PP\big\{\tau^{\eps,\delta}_{x, i}\leq T+T_\theta\big\}\geq\dfrac18\exp\Big(-\dfrac\Delta\delta\Big)\,\forall\,(x,i)\in N\times\M.$$
Let $H^\Delta_2:=8(T+T_\theta).$
Then
it  follows from Lemma \ref{lm3.1} that when $\eps<\eps_2(\Delta)$ for some positive $\eps_2(\Delta)$,
\begin{equation}
\PP\Big\{\tau^{\eps,\delta}_{x, i}\leq H^\Delta_2\exp(\dfrac\Delta\delta)  \Big\}\geq {1\over 2}\,\forall\,(x,i)\in N\times\M.
\end{equation}
\end{proof}

\begin{lm}\label{lm3.4}
Suppose that $\lim\limits_{\eps\to0}\dfrac\delta\eps=\infty$. Moreover, $\sigma(0, i^*)\ne 0$. For every $\Delta>0$, there is
 an $H^\Delta_3>0$ and an $\eps_3(\Delta)>0$ such that for  $\eps<\eps_3(\Delta)$,
$$\PP\Big\{\tau^{\eps, \delta}_{x,i}<H^\Delta_3\exp\big(\dfrac{\Delta}\delta\big)\Big\}>\frac12\,\forall\,(x,i)\in N\times\M.$$
\end{lm}

\begin{proof}
Let $a_2, M(t), T, \zeta_1$ be as in the proof of Lemma \ref{lm3.3}.
We have  $$\PP\big\{\zeta_1\leq T\big\}\geq 1-\exp\Big(-\dfrac{a_3}\eps\Big) \ \hbox{ for some  } \ a_3>0.$$
Since $\bar f(0)=0$, we can apply the large deviation principle (see \cite{HYZ}) to show that
there is $\kappa=\kappa(\Delta)>0$ satisfying
$$\PP\big\{A\big\}\geq 1-\exp\Big\{-\dfrac{\kappa}\eps\Big\}\mbox{ where } A=\left\{\left|\int_0^{u}f(0,\alpha^\eps(s))ds\right|<\dfrac{\Delta}2, \ \forall  0\leq u\leq  T\right\}.$$
We can choose $\theta=\theta(\Delta)>0$ such that $\{y:|y|\leq\theta\}\subset N$ and that $$2\theta+T\times\sup\limits_{|x|<\theta}\{|f(x, i)-f(0, i)|\}<\dfrac{\Delta}2.$$
Note that
\bea
M(1) \ad =\int_0^{\zeta_1}\beta'\sigma(X^{\eps,\delta}_{x, i}(s), \alpha^\eps(s))dW(s)\\
\ad \leq |\beta'X^{\eps, \delta}_{x, i}(\zeta_1)|+|\beta'x|+\Big|\int_0^{\zeta_1}\beta' f(0,\alpha^\eps(s))ds\Big|\\
\aad \ +\int_0^{\zeta_1}\big|\beta'\big(f(X^{\eps, \delta}_{x, i}(s),\alpha^\eps(s))-f(0,\alpha^\eps(s))\big)\big|ds.
\eea
Using arguments similar to those in the proof of Lemma \ref{lm3.3}, we can prove that
$$ \{\eta^\theta_{x, i}\leq T\}\supset\{\sqrt{\delta}M(1)>\Delta\}\cap A\,\forall\,|x|<\theta, i\in\M.$$
Since $\PP\{\delta M(1)>\Delta\}\geq\dfrac12\exp\Big(-\dfrac{\Delta}\delta\Big)$,
we have
\begin{align*} \PP\Big\{\eta^\theta_{x, i}<T\Big\}
 &\geq\PP\{\sqrt{\delta}M(1)>\Delta\}-\big(1-\PP(A)\big)\\
& \geq \dfrac12\exp\big(-\dfrac{\Delta}\delta\big)-\exp\Big(-\dfrac{\kappa}\eps\Big)\,\forall\,|x|<\theta, i\in\M.
\end{align*}
It follows from the hypothesis $\lim\limits_{\eps\to0}\dfrac\delta\eps=0$ that
$$\PP\big\{\eta^\theta_{x, i}<T\big\}\geq\dfrac14\exp\Big(-\dfrac{\Delta}\delta\Big)\,\forall\,|x|<\theta, i\in\M$$
for  sufficiently small $\eps$.
The desired result now can be obtained similar to that of the proof of Lemma \ref{lm3.3}.
\end{proof}

\section{Three Cases}\label{sec:4}
This section provides the proofs of the convergence of $\mu^{\eps,\delta}$ for the three cases
given in \eqref{eq:ep-dl}.
We first state the respective assumptions needed.

\begin{asp}\label{aspc1} {\rm
Assume $\lim\limits_{\eps\to0}\dfrac{\delta}\eps=l>0$. Suppose further that for any critical point $x^*$ of $\bar f(x)$, there is an $i^*$ such that either $f(x^*, i^*)\ne 0$ or $\sigma(x^*, i^*)\ne0$. In what follows, we suppose without loss of generality that $l=1$.}
\end{asp}

\begin{asp}\label{aspc2} {\rm
Assume $\lim\limits_{\eps\to0}\dfrac{\delta}\eps=0$. Moreover, for any critical point $x^*$ of $\bar f(x)$, there is an $i^*$ such that $f(x^*, i^*)\ne0$.
}\end{asp}

\begin{asp}\label{aspc3} {\rm
Assume $\lim\limits_{\eps\to0}\dfrac{\delta}\eps=\infty$. Furthermore, for any critical point $x^*$ of $\bar f(x)$, there is an $i^*$ such that $\sigma(x^*, i^*)\ne 0$.
}\end{asp}

Having estimates in Section 3,
we develop techniques in \cite{CH} to
 prove that for any critical point $x^*$, if one of the three above assumption holds, there is a neighborhood $N_{x^*}$ such that
$\lim\limits_{\eps\to0}\mu^{\eps,\delta}(N_{x^*})=0$.
We suppose that $x^*=0$ and $S, N, G , \gamma$ are defined as in Section 3. Assume that $\limsup\limits_{\eps\to0}\mu^{\eps,\delta}(N)>\vartheta>0$.
Let $R$
be so large that $\mu^{\eps,\delta}(B_R)>1-\dfrac{\vartheta}2$ and that $B_{R-\gamma}$ contains all critical points and the limit cycle $\Gamma$.
By (v) of Assumption \ref{asp1}, we can let $\hat T$
be such that
\begin{equation}\label{e4.1}
\begin{aligned}
\bar X_x(\hat T)&\in B_{R-\gamma}\setminus S, \ \forall  x\in B_R\setminus N.
\end{aligned}
\end{equation}
Recall that there is a $k=k(R)>0$ such that for $\eps$ sufficiently small,
\begin{equation}\label{e4.2}
\PP\left\{|X^{\eps, \delta}_{x, i}(\hat T)-\bar X_x(\hat T)|\geq\gamma\right\}<\exp\Big(-\dfrac{k}{\eps+\delta}\Big), \ \forall  |x|\leq R.
\end{equation}
Choose $\Delta<k/2$ and denote $T_\Delta^{\eps,\delta}=H^\Delta_j\exp\big(\dfrac\Delta\eps\big)$ or $T_\Delta^{\eps,\delta}=H^\Delta_j\exp\big(\dfrac\Delta\delta\big)$ $(j=1, 2$ or $3)$ as in Lemma \ref{lm3.2}, Lemma \ref{lm3.3}, or Lemma \ref{lm3.4} depending on the case we are considering. Clearly, $T_\Delta^{\eps,\delta}\exp(-\dfrac{k}{\eps+\delta})\to 0$ as $\eps\to0$.
Let $\widetilde X^{\eps, \delta}(t)$ be the stationary solution, whose distribution is $\mu^{\eps,\delta}$ at every $t$. Let  $\tau^{\eps,\delta}$
 be the first exit time of $\widetilde X^{\eps, \delta}(t)$ from $N$.
Define events
\begin{align*}
K_1^{\eps, \delta}&=\Big\{\widetilde X^{\eps, \delta}(T_\Delta^{\eps,\delta})\in N, \tau^{\eps,\delta}\geq T_\Delta^{\eps,\delta}, \widetilde X^{\eps,\delta}(0)\in N\Big\}\\
K_2^{\eps, \delta}&=\Big\{\widetilde X^{\eps, \delta}(T_\Delta^{\eps,\delta})\in N, \tau^{\eps,\delta}< T_\Delta^{\eps,\delta}, \widetilde X^{\eps,\delta}(0)\in N\Big\}\\
K_3^{\eps, \delta}&=\Big\{\widetilde X^{\eps, \delta}(T_\Delta^{\eps,\delta})\in N, \widetilde X^{\eps,\delta}(0)\in B_R\setminus N\Big\}\\
K_4^{\eps, \delta}&=\Big\{\widetilde X^{\eps, \delta}(T_\Delta^{\eps,\delta})\in N, \widetilde X^{\eps,\delta}(0)\notin B_R\Big\}.
\end{align*}
Note that $\mu^{\eps,\delta}(N)=\sum_{n=1}^4\PP\{K_n^{\eps, \delta}\}.$
We have $$\PP(K_1^{\eps, \delta})\leq \dfrac12\mu^{\eps,\delta}(N) \ \hbox{ and  }
\ \PP(K_4^{\eps, \delta})< \dfrac{\vartheta}2.$$
Now we estimate $\PP(K_3^{\eps, \delta}).$
It follows from \eqref{e4.1} and \eqref{e4.2} that
$\PP\{X^{\eps, \delta}_{x, i}(\hat T)\in B_R\setminus G \}\geq1-\exp(-\dfrac{k}{\eps+\delta})$ if $x\in B_R\setminus N$. Moreover, if $x\in B_R\setminus G $, $|\bar X_x(t)|\geq 2c_2\geq c_2+\gamma, \ \forall  t\geq0$ which implies that
$$\PP\left\{X^{\eps, \delta}_{x, i}(t)\notin N, \ \forall 0\leq t\leq\hat T\right\}\geq1-\exp\Big(-\dfrac{k}{\eps+\delta}\Big).$$
 Using the Markov property, for any $x\in B_R\setminus N, i\in\M$
\begin{equation}\label{e4.3}
\begin{aligned}
\PP\Big\{ X_{x,i}^{\eps, \delta}&(T_\Delta^{\eps,\delta})\in N\Big\}\\
=&\PP\Big\{ X_{x,i}^{\eps, \delta}(T_\Delta^{\eps,\delta})\in N,  X_{x,i}^{\eps,\delta}(\hat T)\notin B_R\setminus G \Big\}\\
&+\sum_{n=2}^{\lf T_\Delta^{\eps,\delta}/\hat T\rf }\PP\Big\{ X_{x,i}^{\eps, \delta}(T_\Delta^{\eps,\delta})\in N,  X_{x,i}^{\eps,\delta}(n\hat T)\notin B_R\setminus G ,  X_{x,i}^{\eps,\delta}(\iota\hat T)\in B_R\setminus G , \iota=1,...,n-1\Big\}\\
&+\PP\Big\{ X_{x,i}^{\eps, \delta}(T_\Delta^{\eps,\delta})\in N,  X_{x,i}^{\eps,\delta}(\iota\hat T)\in B_R\setminus G , \iota=1,...,[T_\Delta^{\eps,\delta}/\hat T]\Big\}\\
\leq& \PP\Big\{ X_{x,i}^{\eps,\delta}(\hat T)\notin B_R\setminus G\Big\}+\sum_{n=2}^{\lf T_\Delta^{\eps,\delta}/\hat T\rf}\PP\Big\{ X_{x,i}^{\eps,\delta}(n\hat T)\notin B_R\setminus G ,  X_{x,i}^{\eps,\delta}((n-1)\hat T)\in B_R\setminus G \}
\\
&+\PP\{ X^{\eps, \delta}_{x, i}(t)\notin N, \ \forall  \lf T_\Delta^{\eps,\delta}/\hat T\rf
\leq t\leq [T_\Delta^{\eps,\delta}/\hat T]+1,  X^{\eps, \delta}_{x, i}([T_\Delta^{\eps,\delta}/\hat T])\in B_R\setminus G \}\\
\leq&\Big([T_\Delta^{\eps,\delta}/\hat T]+1\Big)\exp\Big(-\dfrac{k}{\eps+\delta}\Big)\leq\Big(T_\Delta^{\eps,\delta}/\hat T+1\Big)\exp\Big(-\dfrac{k}{\eps+\delta}\Big),
\end{aligned}
\end{equation}
where $\lf T_\Delta^{\eps,\delta}/\hat T\rf $ denotes the integer part of $T_\Delta^{\eps,\delta}/\hat T$. It  follows from \eqref{e4.3} that
$$\PP(K_3^{\eps, \delta})=\PP\Big\{\widetilde X^{\eps, \delta}(T_\Delta^{\eps,\delta})\in N, \widetilde X^{\eps,\delta}(0)\in B_R\setminus N\Big\}\leq\Big(T_\Delta^{\eps,\delta}/\hat T+1\Big)\exp\Big(-\dfrac{k}{\eps+\delta}\Big)$$
Since $\lim\limits_{\eps\to0}T_\Delta^{\eps,\delta}\exp\Big(-\dfrac{k}{\eps+\delta}\Big)=0$, we have $\PP(K_3^{\eps,\delta})\to0$ as $\eps\to0$.

To estimate $\PP(K_2^{\eps,\delta})$, note that, similar to \eqref{e4.3}, for any $s\leq T_\Delta^{\eps,\delta}$, we have
$$\PP\Big\{ X_{x,i}^{\eps, \delta}(s)\in N\Big\}\leq\Big(T_\Delta^{\eps,\delta}/\hat T+1\Big)\exp\Big(-\dfrac{k}{\eps+\delta}\Big) \forall x\in B_R\setminus N, i\in\M.$$
The strong Markov property yields
\begin{align*}
\PP(K_2^{\eps, \delta})=&\PP\Big\{\widetilde X^{\eps, \delta}(T_\Delta^{\eps,\delta})\in N, \tau^{\eps,\delta}< T_\Delta^{\eps,\delta}, \widetilde X^{\eps,\delta}(0)\in N\Big\}\\
=&\int_0^{T_\Delta^{\eps,\delta}}\PP\{\tau^{\eps,\delta}\in dt\}\left[\sum_{i\in\M}\int_{\partial N}\PP\Big\{ X_{x,i}^{\eps, \delta}(T_\Delta^{\eps,\delta}-t)\in N\Big\}\PP\big\{\alpha^\eps(t)=i, \widetilde X^{\eps, \delta}(t)\in dx\big\}\right]\\
\leq&\Big(T_\Delta^{\eps,\delta}/\hat T+1\Big)\exp\Big(-\dfrac{k}{\eps+\delta}\Big)\to0\text{ as } \eps\to0.
\end{align*}
Then we have $\vartheta<\limsup\limits_{\eps\to0}\mu^{\eps,\delta}(N)\leq \dfrac12\limsup\limits_{\eps\to0}\mu^{\eps,\delta}(N)+0+0+\dfrac{\vartheta}2$. This contradiction implies that $\lim\limits_{\eps\to0}\mu^{\eps,\delta}(N)=0$.

\begin{thm}\label{thm4.1}
Let Assumption {\rm\ref{asp1}} and one of the three assumptions {\rm\ref{aspc1}}, {\rm\ref{aspc2}}, and
{\rm\ref{aspc3}} hold. In each of the three cases, the family of invariant probability measures $\mu^{\eps,\delta}$ converges weakly to the
 measure $\mu^0$ concentrated on $\Gamma$ of $\bar X(t)$ in the sense that for every bounded and continuous function $g(x, i)$, we have
$$\lim\limits_{\eps\to0}\sum_{i=1}^m\int_{\R^d} g(x, i)\mu^{\eps,\delta}(dx, i)=\dfrac1{T_\Gamma}\int_{0}^{T_\Gamma}\bar g(\bar X_y(t))dt,$$
where $T_\Gamma$ is the period of the cycle and $y\in\Gamma$ and $\bar g(x)=\sum_{i\in\M}g(x, i)\nu_i$.
\end{thm}

\begin{proof}
We have already proved that for each critical point $x^*$ of $\bar f$, we can find a neighborhood $N^*$ such that
$\lim\limits_{\eps\to0}\mu^{\eps,\delta}(N^*)=0.$
Using this fact together with Assumption \ref{asp1} and Lemma \ref{lm2.2}, we establish
the desired results in the same way as in the proof of \cite[Theorem 1]{CH}.
\end{proof}

\section{Applications to A Predator-Prey Model}\label{sec:5}
We consider a stochastic predator-prey model with regime switching
 \begin{equation}\label{ex1}
\left\{\begin{array}{lll}d{X^{\eps,\delta}}(t)=X^{\eps,\delta}(t)\varphi\big(X^{\eps,\delta}(t), Y^{\eps,\delta}(t), \alpha^\eps(t))dt+\sqrt{\delta}\lambda(\alpha^\eps(t))X^{\eps,\delta}(t)dW_1(t)\\
d{Y^{\eps,\delta}}(t)=Y^{\eps,\delta}(t)\psi\big(X^{\eps,\delta}(t), Y^{\eps,\delta}(t), \alpha^\eps(t))dt+\sqrt{\delta}\rho(\alpha^\eps(t))Y^{\eps,\delta}(t)dW_2(t),\end{array}\right.
\end{equation}
where
\bea \ad \varphi(x, y, i)=a(i)-b(i)x-yh(x, y, i) \ \hbox{
and }\\
\ad \psi(x, y, i)=-c(i)-d(i)y+f(i)xh(x, y, i), \eea
where
$a, b, c, d, f, \lambda, \rho$ are positive functions defined on $\M$,
 $\delta$ depends on $\eps$ and $\lim\limits_{\eps\to0}\delta=0$,
$W_1(t)$ and $W_2(t)$ are independent Brownian motions, and
$\alpha^\eps$ is a Markov chain with generator $Q/\eps$, which is independent of $(W_1(t), W_2(t))$.
Assume that $h(x, y, i)$ is positive, bounded, and continuous on $\M\times\R^2_+$.
Note that $xh(x, y, i)$ and
$yh(x, y, i)$ are normally called the functional responses of the predator-prey.
For instance, if $h(x, y, i)$ is constant, the model is the classical Lotka-Volterra.
If $$h(x, y, i)=\dfrac{m_1(i)}{m_2(i)+m_3(i)x+m_4(i)y},$$ the functional response is of Beddington-DeAngelis type.
Let $\bar g=\sum g(i)\nu_i$, $g_m=\min\{g(i):i\in\M\}, g_M=\max\{g(i):i\in\M\}$ where $g=a, b, c, d, f, \varphi, \psi$ and
$h_1(x, y)=\sum h(x, y, i)\nu_i$, $h_2(x, y)=\sum f(i)h(x, y, i)\nu_i.$
The existence and uniqueness of a global positive solution to \eqref{ex1} can be proved in the same manner as in \cite{JJ} or \cite{CDN}.
We denote $Z^{\eps,\delta}_{z, i}(t)=(X^{\eps, \delta}_{z, i}(t), Y^{\eps, \delta}_{z, i}(t))$ the solution to \eqref{ex1} with initial value $\alpha^{\eps}(0)=i\in\M, Z^{\eps,\delta}_{z, i}(0)=z\in\R^{2}_+.$
Consider the averaged equation
\begin{equation}\label{ex2}
\left\{\begin{array}{lll}
\disp {d \over dt}{X}(t)=X(t)\bar\varphi(X(t), X(t))
=X(t)\left[\bar a-\bar bX(t)-Y(t)h_1(X(t), Y(t))\right]\\
\disp {d \over dt}
Y(t)=Y(t)\bar\psi(X(t), Y(t)))=Y(t)\left[-\bar c-\bar d Y(t)+X(t)h_2(X(t), Y(t))\right]
.\end{array}\right.
\end{equation}
We denote by $\bar Z_z(t)=(\bar X_z(t), \bar Y_z(t))$ the solution to \eqref{ex2} with initial value $\bar Z_z(0)=z$.

\begin{asp}\label{asp5.1} {\rm
\begin{enumerate}[(i)]$\text{}$
  \item \eqref{ex2} has a finite number of positive equilibria and any positive solution not starting at an equilibrium converges to a stable limit cycle.
  \item $\frac{\bar b}{\bar a}h_2(\frac{\bar b}{\bar a}, 0)>\bar c$.
\end{enumerate}
}\end{asp}

Note that the Jacobian of
$\Big(x\bar\phi(x,y), y\bar\psi(x,y)\Big)^\top$
at $\left(\frac{\bar b}{\bar a}, 0\right)$
has two eigenvalues: $-\bar c+\frac{\bar b}{\bar a}h_2(\frac{\bar b}{\bar a}, 0)$
and $-\frac{\bar b^2}{\bar a}<0$.
If $-\bar c+\frac{\bar b}{\bar a}h_2(\frac{\bar b}{\bar a}, 0)<0$
then $\left(\frac{\bar b}{\bar a}, 0\right)$ is a stable equilibrium of \eqref{ex2}, which violates the condition (i).
Thus
 condition (ii) is often contained in (i).

We can apply theorem \ref{thm4.1} to our model if we can verify (vi) of Assumption \ref{asp1} since the other conditions are clearly satisfied.
Since the process $\alpha^\eps(t)$ is ergodic and the diffusion is nondegenerate, an invariant probability measure of the solution $Z^{\eps,\delta}(t)$ is unique if it exists.
As mentioned in the introduction, it is unlikely to find a Lyapunov-type function satisfying the hypothesis of \cite[Theorem 3.26]{YZ} in order to prove the existence of an invariant probability measure.
Note that the tightness of family of invariant probability measures cannot be proved using the method of \cite{DDT, DDY}. We overcome
 the difficulties by using a new technical tool. Dividing the domain into several parts and then constructing a suitable truncated Lyapunov-type function, we can estimate the average probability that the solution belongs to each part of our partition.
 To be more precise, instead of proving (vi) of Assumption \ref{asp1}, we derive a slightly weaker property, namely, the family of invariant probability measures being eventually tight in the sense that for any $\Delta>0$, there are  $0<\eps_0,\delta_0<1< L$ such that
for all $\eps<\eps_0,\delta<\delta_0$, the unique invariant measure $\mu^{\eps,\delta}$ of $(Z^{\eps,\delta}(t), \alpha^\eps(t))$ satisfying
$$\mu^{\eps,\delta}([ L^{-1}, L]^2)\geq 1-\Delta.$$
It can be seen in the proof in Section \ref{sec:4} that the eventual tightness is sufficient to obtain the weak convergence of $\mu^{\eps,\delta}$ to the  measure on the limit cycle.
To proceed, we first need some auxiliary results.

\begin{lm}\label{lm5.1}
There are $ K_1, K_2>0$ such that for any $0<\eps,\delta<1$ and any $(i_0, z_0)\in\M\times\R^{2,\circ}_+$
$($where $\R^{2,\circ}_+$ denotes the interior of $\R^2_+)$,
we have
$$\dfrac1t\E\int_0^t|Z_{z_0, i_0}^{\eps,\delta}(s)|^2ds\leq  K_1|z_0|+ K_1, \ \forall  t\geq1$$
and
$$\limsup\limits_{t\to\infty}\E|Z_{z_0, i_0}^{\eps,\delta}(t)|^2\leq K_2.$$
\end{lm}

\begin{proof}
Let
$\theta<\min\{f_Mb(i), d(i): i\in\M\}$.
Define $$\hat  K_1=\sup\limits_{(x, y, i)\in\M\times\R^2_+}\{f_Mx(a(i)-b(i))-y(c(i)-d(i)y)+ \theta(x^2+y^2)\}<\infty.$$
Consider $\hat V(x, y, i)=f_M x+y.$
We can  check that $\mathcal{L}^{\eps,\delta}\hat V(x, y, i)\leq\hat  K_1-\theta(x^2+y^2),$
where $\mathcal{L}^{\eps,\delta}$ the operator associated with \eqref{ex1} (see \cite[p. 48]{MY} or \cite{YZ} for the formula of $\mathcal{L}^{\eps,\delta}$).
Similarly, we can verify that there is $\hat K_2>0$ such that for all $\eps<1,\delta<1$,
$\mathcal{L}^{\eps,\delta} (\hat V^2(x, y, i))\leq\hat  K_2-\hat V^2(x, y, i)$.
For each $k >0$,
define the stopping time $\sigma_k=\inf\{t: x(t)+y(t)>k\}.$
By the generalized It\^o  formula for $\hat V(x(t), y(t),\alpha^\eps(t))$.
\begin{equation}\label{ex3}
\begin{aligned}
\E \hat V(Z_{z_0, i_0}^{\eps,\delta}(t\wedge\sigma_k), \alpha^\eps(t\wedge\sigma_k))
&= \hat V(z_0, i_0)+\E\int_0^{t\wedge\sigma_k}\mathcal{L}^{\eps,\delta}\hat V(Z_{z_0, i_0}^{\eps,\delta}(s), \alpha^\eps(s))ds\\
&\leq f_M x_0+y_0+\E\int_0^{t\wedge\sigma_k}\big[\hat K_1-\theta|Z_{z_0, i_0}^{\eps,\delta}(s)|^2\big]ds.
\end{aligned}
\end{equation}
Hence
$$\theta\E\int_0^{t\wedge\sigma_k}|Z_{z_0, i_0}^{\eps,\delta}(s)|^2ds\leq f_M x_0+y_0+\hat K_1t.$$
Letting $k\to\infty$ and dividing both sides by $\dfrac\theta{t}$ we have
\begin{equation}\label{ex3b}
\dfrac1t\E\int_0^t|Z_{z_0, i_0}^{\eps,\delta}(s)|^2ds\leq \dfrac{f_M x_0+y_0}{\theta t}+\dfrac{\hat K_1}\theta.
\end{equation}
Applying the generalized It\^o  formula to $e^{t}\hat V^2(Z_{z_0, i_0}^{\eps,\delta}(t), \alpha^{\eps}(t))$.
\begin{equation}\label{ex3a}
\begin{aligned}
\E e^{t\wedge\sigma_k}&\hat V^2(Z_{z_0, i_0}^{\eps,\delta}(t\wedge\sigma_k), \alpha^{\eps}(t\wedge\sigma_k))\\
&= \hat V^2(z_0, i_0)+\E\int_0^{t\wedge\sigma_k}e^s\big[(\hat V^2(Z_{z_0, i_0}^{\eps,\delta}(s), \alpha^\eps(s))+\mathcal{L}^{\eps,\delta}\hat V^2(Z_{z_0, i_0}^{\eps,\delta}(s), \alpha^\eps(s))\big]ds\\
&\leq(f_M x_0+y_0)^2+\hat K_2\E\int_0^{t\wedge\sigma_k}e^sds\leq (f_M x_0+y_0)^2+\hat K_2 e^t.
\end{aligned}
\end{equation}
Letting $k\to\infty$, then dividing both sides by $e^t$, we have
\begin{equation}\label{ex3c}
\E \big[f_MX_{z_0, i_0}^{\eps,\delta}(t)+Y_{z_0, i_0}^{\eps,\delta}(t)\big]^2\leq (f_M x_0+y_0)^2e^{-t}+\hat K_2.
\end{equation}
The assertions follows directly from \eqref{ex3b} and \eqref{ex3c}.
\end{proof}

\begin{lm}\label{lm5.1a}
There is a $ K_3>0$ such that
$$\dfrac1t\E\int_0^t\Big[ \varphi^2(Z_{z, i}^{\eps,\delta}(s), \alpha^\eps(s))+\psi^2(Z_{z, i}^{\eps,\delta}(s), \alpha^\eps(s))\Big]ds\leq K_3(1+|z|^2)$$
for all $\eps,\delta\in (0, 1], z\in\R^2_+$.
\end{lm}

\begin{proof}
Since $h(x, y, i)$ is bounded, there is a $C>0$ such that
$$\varphi^2(z, i)+\psi^2(z, i)\leq C(1+|z|^2).$$
The claim therefore follows directly from Lemma \ref{lm5.1}.
\end{proof}

Define
$\Upsilon(z,i)=\frac{2\bar c}{\bar a}\varphi(z,i)+\psi(z,i)$
and
$\bar\Upsilon(z)=\frac{2\bar c}{\bar a}\bar\varphi(z)+\bar\psi(z)$.
We have the following lemma.
\begin{lm}\label{lm5.2}
Let $\gamma_0=\dfrac12 \Big(\bar c\wedge \big(-\bar c+\frac{\bar b}{\bar a}h_1(\frac{\bar b}{\bar a}, 0)\big)\Big)>0.$ For any $H>\frac{\bar b}{\bar a}+1$,  there are $T>0$ and $\beta>0$ satisfying
\begin{equation}\label{e0-lm5.2}
\dfrac1{T}\int_0^{T}\bar\Upsilon(\bar Z_z(t))dt\geq\gamma_0, \,\text{ if }\, z=(x,y)\in\R^2_+\,\text{ satisfying }\, x\wedge y\leq\beta,  x\vee y\leq H.
\end{equation}
\end{lm}

\begin{proof}
Since $\lim\limits_{t\to\infty}\bar Z_{(0, y)}(t)\to (0,0), \ \forall  y\in\R_+$
and
\begin{equation}\label{e1-lm5.2}
\bar\Upsilon(0,0)=\frac{2\bar c}{\bar a}\bar\varphi(0,0)+\bar\psi(0,0)=\frac{2\bar c}{\bar a}\bar a- \bar c=\bar c\geq 2\gamma_0,
\end{equation}
there exists $T_1>0$ such that
\begin{equation}\label{e2-lm5.2}
\dfrac1t\int_0^t\bar\Upsilon(\bar Z_{(0, y)}(s))ds\geq\frac32\gamma_0\,\text{
for }\,t\geq T_1,\,y\in[0,H].
\end{equation}

By \eqref{e1-lm5.2} and the continuity of $\bar\Upsilon(\cdot)$, there exists $\beta_1\in (0, \frac{\bar b}{\bar a})$
such that
\begin{equation}\label{e3-lm5.2}
\bar\Upsilon(x,0)\geq \frac74\gamma_0,\text{ if } x\leq\beta_1.
\end{equation}
Since
$$\bar\Upsilon\left(\frac{\bar b}{\bar a},0\right)=\frac{2\bar c}{\bar a}\bar\varphi \left(\frac{\bar b}{\bar a},0\right)+\bar\psi \left(\frac{\bar b}{\bar a},0\right)=-\bar c+\frac{\bar b}{\bar a}h_1 \left(\frac{\bar b}{\bar a},0\right)\geq2\gamma_0$$
and
$$\lim\limits_{t\to\infty}\bar Z_{(x, 0)}(t)\to \left(\frac{\bar b}{\bar a},0\right), \ \forall  x>0,$$
there exists a $T_2>0$ such that
\begin{equation}\label{e5-lm5.2}
\dfrac1t\int_0^t\bar\Upsilon(\bar Z_{(x, 0)}(s))ds\geq\frac74\gamma_0\,\text{
for }\,t\geq T_2,\,x\in[\beta_1,H].
\end{equation}

Let $\bar M_H=\sup_{x\in[0,H]}\left\{|\Upsilon(x,0)|\right\}
$,
$\bar t_{x}=\inf\{t\geq0: X_{x,0}\geq\beta_1\}$
and $T_3=\left(4\frac{\bar M_H}{\gamma_0}+7\right)T_2$
it can be seen from the equation of $\bar X(t)$ that  $\bar X_{(x,0)}(t)\in[\beta_1,H]$
if $t\geq \bar t_x, x\in(0,\beta_1]$.
For $t\geq T_3$, we can use \eqref{e3-lm5.2} and \eqref{e5-lm5.2} to estimate $\frac1t\int_0^t\bar\Upsilon(\bar Z_{(x, 0)}(s))ds$ in
the following three cases.

{\bf Case 1}.  If $t-T_2\leq \bar t_x\leq t$
then
$$
\begin{aligned}
\int_0^t\bar\Upsilon(\bar Z_{(x, 0)}(s))ds
&=
\int_0^{\bar t_x}\bar\Upsilon(\bar Z_{(x, 0)}(s))ds
+\int_{\bar t_x}^t\bar\Upsilon(\bar Z_{(x, 0)}(s))ds\\
&\geq \frac74\gamma_0(t-T_2)-T_2\bar M_H\geq\frac32\gamma_0t,\,\,\bigg(\text{since }\, t\geq \Big(4\frac{\bar M_H}{\gamma_0}+7\Big)T_2\bigg).
\end{aligned}
$$

{\bf Case 2}.  If  $\bar t_x\leq t-T_2$
then
$$
\begin{aligned}
\int_0^t\bar\Upsilon(\bar Z_{(x, 0)}(s))ds
&=
\int_0^{\bar t_x}\bar\Upsilon(\bar Z_{(x, 0)}(s))ds
+\int_{\bar t_x}^t\bar\Upsilon(\bar Z_{(x, 0)}(s))ds\\
&\geq \frac74\gamma_0(t-\bar t_x)+\frac74\gamma_0\bar t_x\geq\frac32\gamma_0t.
\end{aligned}
$$

{\bf Case 3}.  If  $\bar t_x\geq t$
then
$$
\begin{aligned}
\int_0^t\bar\Upsilon(\bar Z_{(x, 0)}(s))ds
&=
\int_0^{\bar t_x}\bar\Upsilon(\bar Z_{(x, 0)}(s))ds
\geq \frac74\gamma_0\bar t_x\geq\frac32\gamma_0t.
\end{aligned}
$$

Thus, in any case, we have
\begin{equation}\label{e6-lm5.2}
\dfrac1t\int_0^t\bar\Upsilon(\bar Z_{(x, 0)}(s))ds\geq\frac32\gamma_0t,\,\text{ if } t\geq T_3, x\in(0,H].
\end{equation}
Let $T=T_1\vee T_3$.
By the continuous dependence
of solutions to initial values,
there is $\beta>0$
such that
\begin{equation}\label{e7-lm5.2}
\dfrac1T\int_0^T\left|\bar\Upsilon(\bar Z_{z_1}(s))-\bar\Upsilon(\bar Z_{z_2}(s))\right|ds\leq \frac12\gamma_0\,\text{ given that }\, |z_1-z_2|\leq\beta, z_1,z_2\in[0,H]^2.
\end{equation}
Combining \eqref{e2-lm5.2}, \eqref{e6-lm5.2} and \eqref{e7-lm5.2}
we obtain the desired result.
\end{proof}

Generalizing the techniques in \cite{DY},
we divide the proof of the eventual tightness into two lemmas.

\begin{lm}\label{lm5.3}
For any $\Delta>0$, there exist $\eps_0>0$, $\delta_0$,   a compact set $\mathcal K\subset\R^{2,\circ}_+$,
and $T>0$ such that
$$\liminf\limits_{k\to\infty}\sum_{n=0}^{k-1}\PP\left\{Z^{\eps,\delta}_{z_0, i_0}(nT)\in \mathcal K\right\}\geq 1-\dfrac\Delta3\, \text{ for any }\, \eps<\eps_0, \delta<\delta_0.$$
\end{lm}

\begin{proof}
For any $\Delta>0$, let $H=H(\Delta)>\frac{\bar b}{\bar a}+1$ be chosen later and define
$D=\{(x, y): 0<x, y\leq H\}$.
Let $T>0$ and $\beta>0$ such that \eqref{e0-lm5.2} is satisfied
and $D_1=\{(x, y): 0<x, y\leq H, x\vee y<\beta\}\subset D$.
Define $V(x, y)=-\frac{2\bar c}{\bar a}\ln x-\ln y+ C$ where $C$ is a positive constant such that $V(z)\geq0\,\forall\, z\in D$. In view of the generalized It\^o  formula,
$$
\begin{aligned}
V(Z_{z, i}^{\eps,\delta}(t))-V(z)=&\int_0^t\left[-\Upsilon\big(Z_{z, i}^{\eps,\delta}(s), \alpha^\eps(s)\big)+\dfrac{\delta}2\left(\frac{2\bar c}{\bar a}\lambda^2(\alpha^\eps(s))+\rho^2(\alpha^\eps(s))\right)\right]ds\\
&-\frac{2\bar c}{\bar a}\int_0^t\sqrt{\delta}\lambda(\alpha^\eps(s))dW_1(s)-\int_0^t\sqrt{\delta}\rho(\alpha^\eps(s))dW_2(s).
\end{aligned}
$$
For $A\in\F$, using Holder's inequality and the Burkholder-Davis-Gundy inequality, we have	 	
\begin{equation}\label{ex11}
\begin{aligned}\E\Big(\1_{A}&\big|V(Z_{z, i}^{\eps,\delta}(T))-V(z)\big|\Big)\\
\leq&\left|\E\1_A\int_0^{T}\Upsilon\big(Z_{z, i}^{\eps,\delta}(t), \alpha^\eps(t)\big)dt\right|+\E\1_A\int_0^{T}\dfrac{\delta}2\left(\frac{2\bar c}{\bar a}\lambda^2(\alpha^\eps(t))+\rho^2(\alpha^\eps(t))\right)dt\\&+\frac{2\bar c}{\bar a}\E\1_A\left|\int_0^{T}\sqrt{\delta}\lambda(\alpha^\eps(t))dW_1(t)\right|+\E\1_A\left|\int_0^{T}\sqrt{\delta}\rho(\alpha^\eps(t))dW_2(t)\right|\\
\leq&T(\E\1_A)^{\frac12}\E\int_0^{T}\left[\Upsilon\big(Z_{z, i}^{\eps,\delta}(t), \alpha^\eps(t)\big)+\dfrac\delta2\left(\frac{2\bar c}{\bar a}\lambda^2(\alpha^\eps(t))+\rho^2(\alpha^\eps(t))\right)\right]dt\\
&+\delta\PP(A)\E\int_0^T\left(\frac{2\bar c}{\bar a}\lambda^2(\alpha^\eps(t))+\rho^2(\alpha^\eps(t))\right)dt\\
\leq&  K_4T(1+|z|)\sqrt{\PP(A)}
\end{aligned}
\end{equation}
where the last inequality follows from \eqref{lm5.1a} and the boundedness of $\rho(i)$ and $\lambda(i)$.
If $A=\Omega$, we have
\begin{equation}\label{ex11a}
\dfrac1{T}\E\Big(\big|V(Z_{z, i}^{\eps,\delta}(T))-V(z)\big|\Big)\leq  K_4(1+|z|).
\end{equation}
Let $\hat H_{T}>H$ such that $\bar X_z(t)\vee \bar Y_z(t)\leq \hat H_{T}$ for all $z\in[0,H]^2,\, 0\leq t\leq {T}$ and
 $$\bar{d}_H=\sup\left\{\left|\dfrac{\partial \bar\Upsilon}{\partial x}(x, y)\right|, \left|\dfrac{\partial \bar\Upsilon}{\partial y}(x, y)\right|: (x,y)\in\R^2_+, x\vee y\leq \hat H_{T}\right\}.$$
Lemma \ref{lm2.2} implies that there are $\delta_0,\eps_0$ such that if $\eps<\eps_0,\delta<\delta_0$,
\begin{equation}\label{ex5}
\PP\left\{|\bar X_z(t)-X_{z, i}^{\eps,\delta}(t)|+|\bar Y_z(t)-Y_{z, i}^{\eps,\delta}(t)|<1\wedge\dfrac{\gamma_0}{2\bar{d}_H}, \ \forall  t\in[0,{T}]\right\}>1-\dfrac\varsigma6, \ \forall \,z\in \bar D.
\end{equation}
On the other hand, if $|\bar X_z(t)-X_{z, i}^{\eps,\delta}(t)|+|\bar Y_z(t)-Y_{z, i}^{\eps,\delta}(t)|<1\wedge\dfrac{\gamma_0}{2\bar d_H}$, we have
\begin{equation}\label{ex6}
\begin{aligned}
\bigg|\dfrac{1}{T}\int_0^{T}&\Upsilon(Z_{z, i}^{\eps,\delta}(t), \alpha^\eps(t))dt
- \dfrac{1}{T}\int_0^{T}\bar\Upsilon(\bar Z_{z, i}(t))dt\bigg|\\
\leq&\dfrac1{T}\left|\int_0^{T}\Big(\bar\Upsilon(Z_{z, i}^{\eps,\delta}(t))-\bar\Upsilon(\bar Z_{z, i}(t))\Big)dt\right|\\
& +\dfrac1{T}\left|\int_0^{T}\Big(\Upsilon(Z_{z, i}^{\eps,\delta}(t), \alpha^\eps(t))-\bar\Upsilon(Z_{z, i}^{\eps,\delta}(t))\Big)dt\right|\\
\leq&\dfrac{\gamma_0}2+\dfrac{F_H}{T}\bigg|\int_0^{T}\sum_{j\in\M}\big(\1_{\{\alpha^\eps(t)=j\}}-v_j\big)dt\bigg|
\end{aligned}
\end{equation}
where
$F_H=\sup\{|\Upsilon(z, i)| i\in\M, z\in[0, K_{T}+1]^2\}.$
In view of \cite[Lemma 2.1]{HYZ},
\begin{equation}\label{ex7}
\E\bigg|\dfrac1T\int_0^{T}\sum_{j\in\M}\big(\1_{\{\alpha^\eps(t)=j\}}-v_j\big)dt\bigg|^2\leq \dfrac{\kappa}{T}\eps\text{ for some constant }\kappa>0.
\end{equation}	
On the one hand,
\begin{equation}\label{ex9}
\E\dfrac1{T}\left|\int_0^{T}\Big(-\frac{2\bar c}{\bar a}\lambda(\alpha^\eps(t))dW_1(t)+\rho(\alpha^\eps(t))dW_2(t)\Big)\right|^2\leq\frac{4\bar c^2}{\bar a^2}\lambda_M^2+\rho_M^2.
\end{equation}

Combining \eqref{e0-lm5.2}, \eqref{ex5}, \eqref{ex6}, \eqref{ex7}, and \eqref{ex9}, we can reselect $\eps_0$ and $\delta_0$ such that for $\eps<\eps_0,\delta<\delta_0$ we obtain \eqref{ex8}--\eqref{ex8a} as below.
\begin{equation}\label{ex8}
\PP\left\{\dfrac{-1}{T}\int_0^{T}\Upsilon(\alpha^\eps(t), Z_{z, i}^{\eps,\delta}(t))dt\leq-0.5\gamma_0\right\}\geq 1-\dfrac\varsigma3\forall\,z\in D_1, i\in\M,
\end{equation}
\begin{equation}\label{ex8b}
\PP\Big\{X_{z, i}^{\eps,\delta}(T)\vee Y_{z, i}^{\eps,\delta}(T)\leq H \mbox{ (or equivalently } Z_{z, i}^{\eps,\delta}(T)\in D)\Big\}\geq1-\dfrac\varsigma3
,\end{equation}
and
\begin{equation}\label{ex8a}
\PP\left\{\delta\vartheta+\dfrac{\sqrt{\delta}}{T}\left|\int_0^{T}\Big(\frac{2\bar c}{\bar a}\lambda(\alpha^\eps(t))dW_1(t)+\rho(\alpha^\eps(t))dW_2(t)\Big)dt\right|<0.25\gamma_0\right\}>1-\dfrac\varsigma3
\end{equation}
where
$\vartheta=\frac12\left(\frac{2\bar c}{\bar a}\lambda_M^2+\rho_M^2\right).$
Consequently, for any $(z, i)\in D_1\times\M$, there is a subset $\Omega_{z, i}^{\eps,\delta}\subset\Omega$ with $\PP(\Omega_{z, i}^{\eps,\delta})\geq1-\varsigma$
in which we have $Z_{z, i}^{\eps,\delta}(T)\in D$ and
\begin{equation}\label{ex10}
\begin{aligned}
\dfrac1{T}\big(V(Z_{z, i}^{\eps,\delta}({T}))-V(z)\big)\leq&\dfrac{-1}{T}\int_0^{T}\Upsilon(\alpha^\eps(t), Z_{z, i}^{\eps,\delta}(t))dt+\delta\vartheta\\&+\dfrac1{T}\Big|\int_0^{T}\sqrt{\delta}\Big(\frac{2\bar c}{\bar a}\lambda(\alpha^\eps(t))dW_1(t)+\rho(\alpha^\eps(t))dW_2(t)\Big)\Big|\\
\leq& -0.25\gamma_0
\end{aligned}
\end{equation}
On the other hand, we deduce from \eqref{ex11a} that for $z\in D$
\begin{equation}\label{ex10a}
\PP\left\{\dfrac1{T}\big(V(Z_{z, i}^{\eps,\delta}({T}))-V(z)\big)\leq\Lambda:=\frac{ K_4(1+2H)}\varsigma\right\}\geq 1-\varsigma
\end{equation}
Moreover, it also follows from \eqref{ex11a} that for $z\in D\setminus D_1$
$$\E V(Z_{z, i}^{\eps,\delta}({T})\leq \sup_{z\in D\setminus D_1}\big(V(z)+ K_4{T}|z|\big).$$
Let
\begin{equation}\label{ex10d}
L_1=\sup_{z\in D\setminus D_1}V(z)+\Lambda T,\,\text{ and } L_2=L_1+0.25\gamma_0.
\end{equation}
Define $D_2=\{(x, y)\in\R^{2,\circ}_+: (x, y)\in D, V(x, y)>L_2\}$.
 Define $U(z)=V(z)\vee L_1.$ It is clear that
\begin{equation}\label{ex12a}
U(z_2)-U(z_1)\leq|V(z_2)-V(z_1)| \text{ for any }z_1, z_2\in\R^{2\circ}_+.
\end{equation}
It follows from  \eqref{ex11}
that for any $\delta,\eps<1$ and $A\in\F$, $z\in D$ we have
\begin{equation}\label{ex12}
\dfrac1{T}\E\1_{A}\Big|V(Z_{z, i}^{\eps,\delta}(T))-V(z)\Big|\leq K_4(2H+1)\sqrt{\PP(A)}
.\end{equation}
Applying \eqref{ex12} and \eqref{ex12a} for $A=\Omega\setminus\Omega_{z, i}^{\eps,\delta}$
we have
\begin{equation}\label{ex12b}
\dfrac1{T}\E\1_{\Omega\setminus\Omega_{z, i}^{\eps,\delta}}\Big[U(Z_{z, i}^{\eps,\delta}(T))-U(z)\Big]\leq K_4(2H+1)\sqrt{\varsigma},\text{ if } z\in D_1.
\end{equation}
In view of \eqref{ex10}, for
 $z\in D_2\text{ we have }V\big(Z_{z, i}^{\eps,\delta}(T)\big)<V(z)-0.25\gamma_0T.$
 By the definition of $D_2$,
 we also have $L_1\leq V(z)-0.25\gamma_0T.$
Thus, for $z\in D_2$ and $\omega\in\Omega^{\eps,\delta}_{z, i}$, we have
$$U\big(Z_{z, i}^{\eps,\delta}(T)\big)=L_1\vee V\big(Z_{z, i}^{\eps,\delta}(T)\big)\leq V(z)-0.25\gamma_0T=U(z)-0.25\gamma_0T,$$
which implies
\begin{equation}\label{ex10c}
\dfrac1{T}\Big[\E\1_{\Omega_{z, i}^{\eps,\delta}}U(Z_{z, i}^{\eps,\delta}({T}))-\E\1_{\Omega_{z, i}^{\eps,\delta}}U(z)\Big]\leq-0.25\gamma_0\PP(\Omega_{z, i}^{\eps,\delta})\leq -0.25\gamma_0(1-\varsigma).
\end{equation}

Combining \eqref{ex12b} with \eqref{ex10c}, we obtain
\begin{equation}\label{ex13}
\dfrac1{T}\Big[\E U(Z_{z, i}^{\eps,\delta}({T}))-U(z)\Big]\leq-0.25\gamma_0(1-\varsigma)+ K_4(2H+1)\sqrt{\varsigma}, \ \forall  z\in D_2.
\end{equation}
For $z\in D_1\setminus D_2$,
 and $\omega\in\Omega_{z, i}^{\eps,\delta}$, we have from \eqref{ex10}  that $ V(Z_{z, i}^{\eps,\delta}({T}))\leq V(z)$, from which we derive
$U(Z_{z, i}^{\eps,\delta}({T}))=L_1\vee V(Z_{z, i}^{\eps,\delta}({T}))\leq U(z)=V(z)\vee L_1.$
Hence, for $z\in D_1\setminus D_2$, $U(Z_{z, i}^{\eps,\delta}({T}))-U(z)\leq 0$ in $\Omega_{z, i}^{\eps,\delta}$.
This and \eqref{ex12b} imply
\begin{equation}\label{ex14}
\dfrac1{T}\Big[\E U(Z_{z, i}^{\eps,\delta}({T}))-U(z)\Big]\leq K_4(2H+1)\sqrt{\varsigma}, \ \forall  z\in D_1\setminus D_2.
\end{equation}
If $z\in D\setminus D_1$, $U(z)=L_1$ and we have from \eqref{ex10a} and \eqref{ex10d} that
$$\PP\big\{U(Z_{z, i}^{\eps,\delta}({T}))=L_1\big\}= \PP\big\{V(Z_{z, i}^{\eps,\delta}({T}))\leq L_1\big\}\geq 1-\varsigma.$$
Thus
$$\PP\{U(Z_{z, i}^{\eps,\delta}({T}))=U(z)\}\geq 1-\varsigma.$$
Use \eqref{ex12} and \eqref{ex12a} again to arrive at
\begin{equation}\label{ex16}
\dfrac1{T}\Big[\E U(Z_{z, i}^{\eps,\delta}({T}))-U(z)\Big]\leq K_4(2H+1)\sqrt{\varsigma}, \ \forall \,z\in D\setminus D_1.
\end{equation}
On the other hand, we deduce from \eqref{ex11a} and \eqref{ex12a} that
\begin{equation}\label{ex17}
\dfrac1{T}\Big[\E U(Z_{z, i}^{\eps,\delta}({T}))-U(z)\Big]\leq  K_4(1+|z|), \ \forall  z\in\R^{2,\circ}_+.
\end{equation}
Now, choose  arbitrary $(z_0, i_0)$ in $\R^{2,\circ}_+\times\M$,
it follows from the Markov property that
$$
\begin{aligned}
\dfrac1{T}\Big[&\E U(Z_{z_0, i_0}^{\eps,\delta}((n+1){T}))-\E U(Z_{z_0, i_0}^{\eps,\delta}({nT}))\Big]\\
&=\sum_{i\in\M}\int_{\R^{2,\circ}_+}\dfrac1{T}\Big[\E U(Z_{z, i}^{\eps,\delta}({T}))-U(z)\Big]\PP\Big\{Z_{z_0, i_0}^{\eps,\delta}(n{T})\in dz, \alpha^\eps(t)=i\Big\}.
\end{aligned}
$$
Subsequently, using \eqref{ex13}, \eqref{ex14},  \eqref{ex16}, and \eqref{ex17}, we obtain
$$
\begin{aligned}
\dfrac1{T}\Big[\E& U(Z_{z_0, i_0}^{\eps,\delta}((n+1){T}))-\E U(Z_{z_0, i_0}^{\eps,\delta}({nT}))\Big]\\
\leq&-\big[0.25\gamma_0(1-\varsigma)- K_4(2H+1)\sqrt{\varsigma}\big]\PP\big\{Z_{z_0, i_0}^{\eps,\delta}(n{T})\in D_2\big\}\\
&  + K_4(2H+1)\sqrt{\varsigma}\PP\big\{Z_{z_0, i_0}^{\eps,\delta}(n{T})\in D\setminus D_2\big\}\\
&+  K_4	\E\1_{\{Z_{z_0, i_0}^{\eps,\delta}(n{T})\notin D\}}\big(1+|Z_{z_0, i_0}^{\eps,\delta}(n{T})|\big)\\
\leq&-0.25\gamma_0(1-\varsigma)\PP\big\{Z_{z_0, i_0}^{\eps,\delta}(n{T})\in D_2\big\}+ K_4(2H+1)\sqrt{\varsigma}\\
&+  K_4\PP\big\{Z_{z_0, i_0}^{\eps,\delta}(nT)\notin D\big\}\E\big(1+|Z_{z_0, i_0}^{\eps,\delta}(nT)|\big)
\end{aligned}
$$
Note that
\bea \ad \liminf\limits_{k\to\infty}\dfrac1k\sum_{n=0}^{k-1}\dfrac1{T}\Big[\E U(Z_{z_0, i_0}^{\eps,\delta}((n+1){T}))-\E U(Z_{z_0, i_0}^{\eps,\delta}({nT}))\Big]\\
\aad \
=\liminf\limits_{k\to\infty}\dfrac1{kT}\E U(Z_{z_0, i_0}^{\eps,\delta}(k{T}))\geq0,
\eea which implies
\begin{equation}\label{ex17a}
\begin{aligned}
0.25\gamma_0(1-\varsigma)&\limsup\limits_{k\to\infty}\dfrac1k\sum_{n=0}^{k-1}\PP\big\{Z_{z_0, i_0}^{\eps,\delta}(n{T})\in D_2\big\}\\
\leq & K_4(2H+1)\sqrt{\varsigma}+ K_3\limsup\limits_{k\to\infty}\dfrac1k\sum_{n=1}^k\PP\big\{Z_{z_0, i_0}^{\eps,\delta}(nT)\notin D\big\}\E\big(1+|Z_{z_0, i_0}^{\eps,\delta}(nT)|\big).
\end{aligned}
\end{equation}
In view of Lemma \ref{lm5.1}, we can choose $H=H(\Delta)$ independent of $(z_0, i_0)$ such that
\begin{equation}\label{ex18}
\limsup\limits_{t\to\infty}\PP\big\{Z_{z_0, i_0}^{\eps,\delta}(t)\notin D\big\}\leq \limsup\limits_{t\to\infty}\dfrac{\E(|Z_{z_0, i_0}^{\eps,\delta}(t)|)}{H}\leq\dfrac{\Delta}{6}
\end{equation}
and
\bea\ad  K_4\limsup\limits_{t\to\infty}\PP\big\{Z_{z_0, i_0}^{\eps,\delta}(t)\notin D\big\}\E\big(1+|Z_{z_0, i_0}^{\eps,\delta}(t)|\big)\\
\aad \ \leq  K_4\limsup\limits_{t\to\infty}\dfrac{\Big[\E\big(1+|Z_{z_0, i_0}^{\eps,\delta}(t)|\big)\Big]^2}{H}\leq\dfrac{0.1\gamma_0}{6}\Delta.
\eea
Hence, we have
\begin{equation}\label{ex19}
 K_4\limsup\limits_{k\to\infty}\dfrac1k\sum_{n=1}^k\PP\big\{Z_{z_0, i_0}^{\eps,\delta}(nT)\notin D\big\}\E\big(1+|Z_{z_0, i_0}^{\eps,\delta}(nT)|\big)\leq\dfrac{0.1\gamma_0}{6}\Delta.
\end{equation}
Choose $\varsigma=\varsigma(H)>0$ such that $0.25\gamma_0(1-\varsigma)\geq0.2\gamma_0$ and $ K_4(2H+1)\sqrt{\varsigma}\leq\dfrac{0.1\gamma_0}{6}\Delta$
and let $\eps_0=\eps_0(\varsigma, H),\delta_0(\varsigma, H)$ such that \eqref{ex8}, \eqref{ex8b}, and \eqref{ex8a} hold.
Thus, it follows from \eqref{ex17a} and \eqref{ex19} that
\begin{equation}\label{ex20}
\limsup\limits_{k\to\infty}\dfrac1k\sum_{n=1}^k\PP\big\{Z_{z_0, i_0}^{\eps,\delta}(nT)\in D_2\big\}\leq \dfrac{\Delta}{6}.
\end{equation}
Therefore, it follows from \eqref{ex18} and \eqref{ex20} that for any $\eps<\eps_0, \delta<\delta_0$, we have
\begin{equation}\label{ex21}
\liminf\limits_{k\to\infty}\dfrac1k\sum_{n=1}^k\PP\big\{Z_{z_0, i_0}^{\eps,\delta}(nT)\in D\setminus D_2\big\}\geq1- \dfrac{\Delta}3.
\end{equation}
The lemma is proved by noting that the set $D\setminus D_2$ is a compact subset of $\R^{2,\circ}_+$.
\end{proof}

\begin{lm}\label{lm5.4}
There are $ L>1$, $\eps_1=\eps(\Delta)>0$, and $\delta_1=\delta_1(\Delta)>0$ such that as long as $0<\eps<\eps_1, 0<\delta<\delta_1$, we have
$$\liminf\limits_{T\to\infty}\dfrac1T\int_0^T\PP\{Z^{\eps,\delta}_{z_0, i_0}(t)\in [ L^{-1}, L]^2\}\geq 1-\Delta, \ \forall  (z_0, i_0)\in\R^{2,\circ}_+\times\M.$$
\end{lm}

\begin{proof}
Let $D$ and $T$ as in Lemma \ref{lm5.3}.
Since $D\setminus D_2$ is a compact set in $\R^{2,\circ}_+$, by a modification of the proof of \cite[Theorem 2.1]{JJ}, we can  show that there is a positive constant $ L>1$ such that
$\PP\{Z_{z, i}(t)\in [ L^{-1}, L]^2\}>1-\dfrac\Delta3, \ \forall  z\in D\setminus D_2, i\in M, 0\leq t\leq T$.
Hence, it follows from the Markov property of the solution that
\bea \ad\PP\{Z_{z_0, i_0}^{\eps, \delta}(t)\in[ L^{-1}, L]^2\}\\
\aad\quad \geq \Big(1-\dfrac{\Delta}3\Big)\PP\big\{Z_{z_0, i_0}^{\eps,\delta}(jT)\in D\setminus D_2\big\}\forall t\in[jT, jT+T].\eea
Consequently,
\begin{align*}
\liminf\limits_{k\to\infty}\dfrac1{kT}&\int_{0}^{kT}\PP\big\{Z_{z_0, i_0}^{\eps,\delta}(t)\in [ L^{-1}, L]^2\big\}\\&\geq \Big(1-\dfrac{\Delta}3\Big)\liminf\limits_{k\to\infty}\dfrac1{k}\sum_{j=0}^{k-1}\PP\big\{Z_{z_0, i_0}^{\eps,\delta}(jT)\in D\setminus D_2)\big\}\geq 1-\Delta,
\end{align*}
It is readily seen from this estimate that
$$\liminf\limits_{T\to\infty}\dfrac1{T}\int_{0}^{T}\PP\big\{Z_{z_0, i_0}^{\eps,\delta}(t)\in [ L^{-1}, L]^2\big\}dt\geq 1-\Delta.$$
\end{proof}
The conclusion of Lemma \ref{lm5.4}
 is sufficient for the existence of a unique invariant probability measure
 $\mu^{\eps,\delta}$ in $\R^{2,\circ}_+\times\M$ of $(Z^{\eps,\delta}(t), \alpha^\eps(t))$ (see \cite{LB} or \cite{MT}). Moreover, the empirical measure
 $\dfrac1t\int_0^t\PP\big\{Z^{\eps,\delta}_{z_0,i_0}(s)\in \cdot\big\}ds$
 converges weakly to the invariant measure $\mu^{\eps,\delta}$ as $t\to\infty$.
Applying Fatou's lemma to the above estimate yields
$$\mu^{\eps,\delta}([ L^{-1}, L]^2)\geq \Delta, \ \forall \,\eps<\eps_0, \delta<\delta_0.$$
The eventual tightness obtained above enables us to conclude the following result.

\begin{thm}
Under Assumption {\rm\ref{asp5.1}}, the process $(X^{\eps,\delta}(t), Y^{\eps,\delta}(t), \alpha^{\eps}(t))$ satisfying \eqref{ex1} has  a unique invariant probability measure $\mu^{\eps,\delta}$ for sufficiently small $\delta$ and $\eps$. Moreover the invariant measure $\mu^{\eps,\delta}$ converges weakly to the measure of the limit cycle of \eqref{ex2} as $\eps\to0$ in the sense
 of Theorem {\rm\ref{thm4.1}} if $\lim\limits_{\eps\to0}\dfrac\delta\eps=l\in(0,\infty]$. In case $\lim\limits_{\eps\to0}\dfrac\delta\eps=0$, to obtain the same conclusion, we need an additional condition that at each critical point $(x^*, y^*)$ of $(\bar\varphi(x, y), \bar\psi(x, y))$, there is $i^*\in\M$ such that either $\varphi(x^*, y^*, i^*)\ne0$ or $\psi(x^*, y^*, i^*)\ne0$.
\end{thm}

\newcommand{\xx}{x^{\eps,\delta}}
\newcommand{\yy}{y^{\eps,\delta}}

\section{An Example}
\label{sec:6}
This section further demonstrates our results by providing an
example.
We consider the following stochastic predator-prey model with Holling functional response in a switching regime
\begin{equation}\label{nex1}
\left\{\begin{array}{ll} \disp d\xx (t)&\!\!\! \disp =\bigg[r(\alpha^\eps(t))\xx(t)\Big(1-\dfrac{\xx(t)}{K(\alpha^\eps(t)}\Big)-\dfrac{m(\alpha^\eps(t))\xx(t)\yy(t)}{a(\alpha^\eps(t))
+b(\alpha^\eps(t))}\bigg]dt\\
& \quad \ +\sqrt{\delta}\lambda(\alpha^\eps(t))\xx(t)dW_1(t)\\
\disp d\yy(t)&\!\!\! \disp =\yy(t)\bigg[-d(\alpha^\eps(t))+\dfrac{e(\alpha^\eps(t))
m(\alpha^\eps(t))\xx(t)}{a(\alpha^\eps(t))+b(\alpha^\eps(t))}-f(\alpha^\eps(t))\yy(t)\bigg]dt\\
&\quad \ +\sqrt{\delta}\rho(\alpha^\eps(t))\xx(t)dW_2(t).\end{array}\right.
\end{equation}
In \eqref{nex1}, $W_1$ and $W_2$ are two independent Brownian motions, $\alpha^\eps(t)$ is a Markov chain
 independent of the Brownian motions and having  state space $\M=\{1, 2\}$ and generator $Q/\eps$, where
$$Q =
\left( \begin{array}{rr}  -1 &  1 \\
 1  &  -1 \\ \end{array} \right),$$
$r(1)=0.9, r(2)=1.1, K(1)=4.737, K(2)=5.238, m(1)=1.2, m(2)=0.8, a(1)=a(2)=1, b(1)=b(2)=1, d(1)=0.85, d(2)=1.15, e(1)=1.5, e(2)=2, f(1)=0.03, f(2)=0.01, \lambda(1)=1, \lambda(2)=2, \rho(1)=3, \rho(2)=1$.
As $\eps$ and $\delta$ tend to 0, a solution of equation \eqref{nex1} converges  to the corresponding solution of
\begin{equation}\label{nex2}
\left\{\begin{array}{lll}\disp {d \over dt} {x}(t)=x(t)\Big(1-\dfrac{x(t)}{5}\Big)-\dfrac{x(t)y(t)}{1+x(t)},
\\
\disp {d \over dt} {y}(t)=y(t)\Big(-1-\dfrac{1.6x(t)}{1+x(t)}-0.02y(t)\Big)
\end{array}\right.
\end{equation}
on each finite time interval.
Using \cite[Theorem 2.6]{SR}, the solution of equation
\eqref{nex2} has a unique limit cycle attracting all positive solution except for
 $(x^*, y^*)=(1.836, 1.795)$, which is a unique equilibrium of \eqref{nex2}.
Moreover, it is easy to check that the drift
 \begin{equation}\label{nex4}
\left(\begin{array}{l}r(i)x(t)\Big(1-\dfrac{x(t)}{K(i)}\Big)-\dfrac{m(i)x(t)y(t)}{a(i)+b(i)}\\
y(t)\Big(-d(i)-\dfrac{e(i)m(i)x(t)}{a(i)+b(i)}\Big)\end{array}\right)
\end{equation}
does not vanish at $(1.836, 1.795)$. Thus, $\mu^{\eps,\delta}$ converges weakly to the stationary distribution of \eqref{nex2} concentrated on the limit cycle as $\eps\to0$.
We illustrate this convergence by some figures showing sample paths of $\eqref{nex1}$ with different values of $(\eps,\delta)$.

\begin{figure}[h]
\centering
\includegraphics[totalheight=2.2in,width=2.1in]{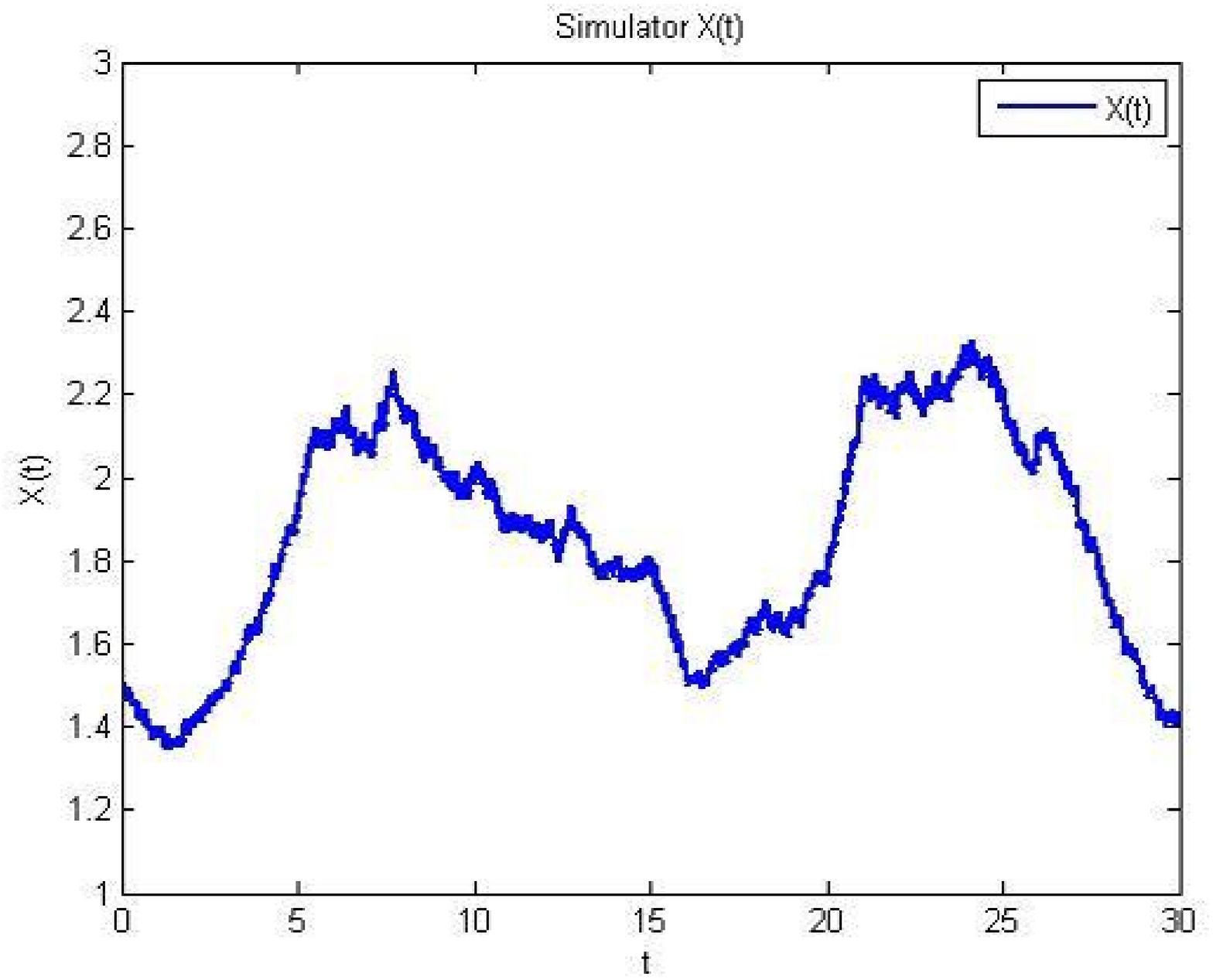}
\includegraphics[totalheight=2.2in,width=2.1in]{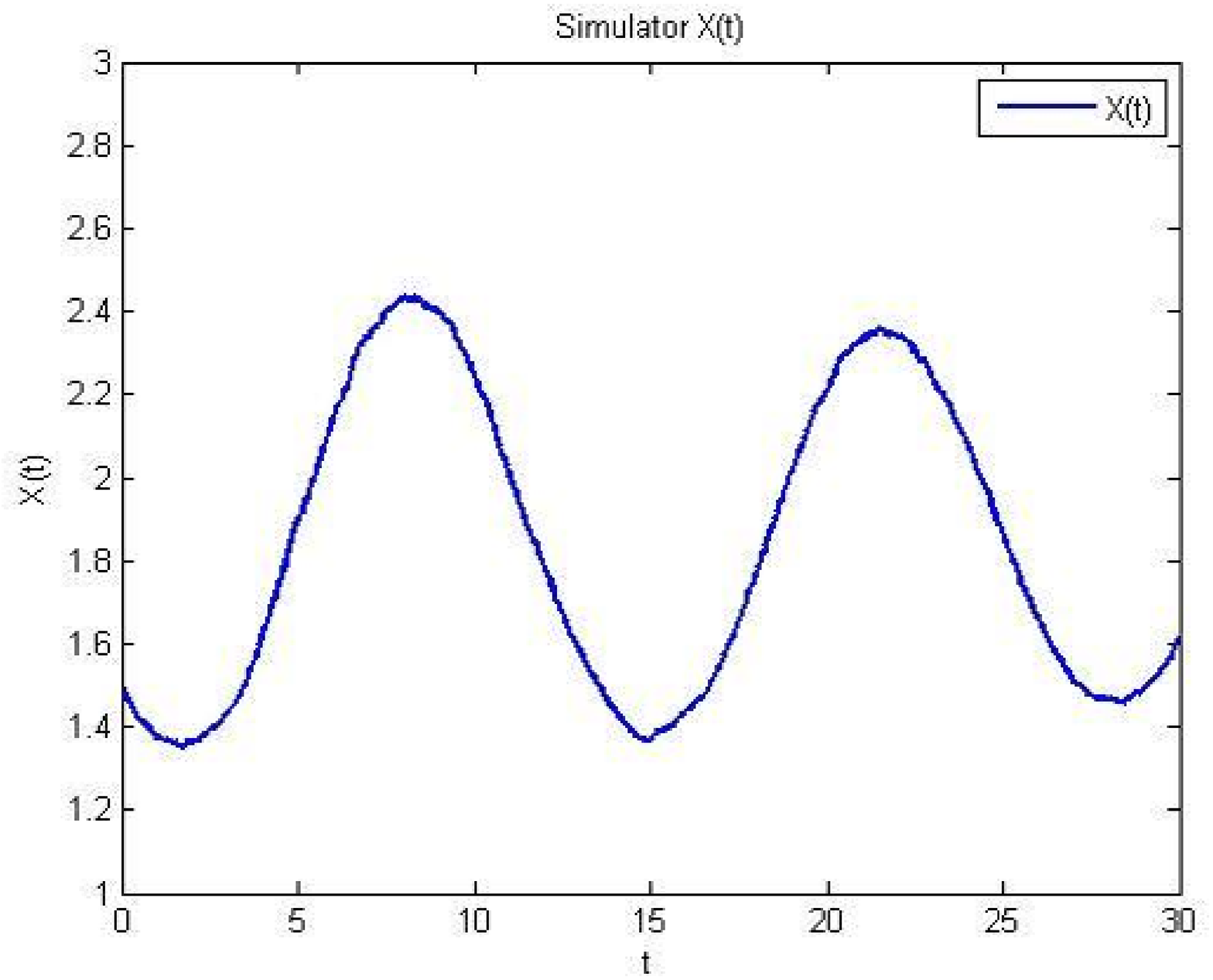}
\includegraphics[totalheight=2.2in,width=2.1in]{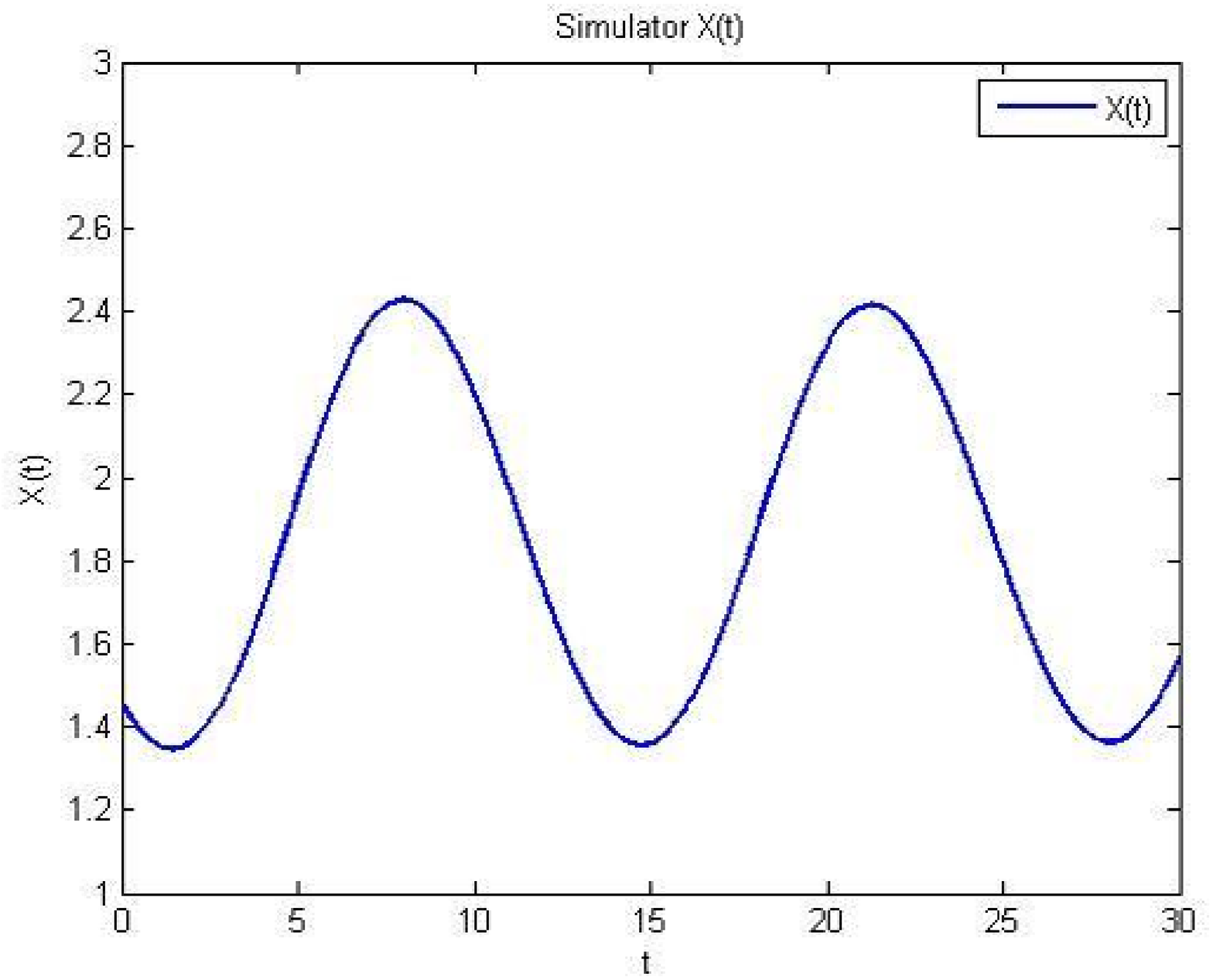}
\caption{From left to right: Graphs of $x^{\eps,\delta}(t)$ of solutions to Eq. \eqref{nex1} with $(\eps,\delta)=(0.01, 0.01)$, $(\eps,\delta)=(0.001, 0.001)$ and $x(t)$ to the averaged Eq. \eqref{nex2} respectively.}
\end{figure}

\begin{figure}[h]
\centering
\includegraphics[totalheight=2.2in,width=2.1in]{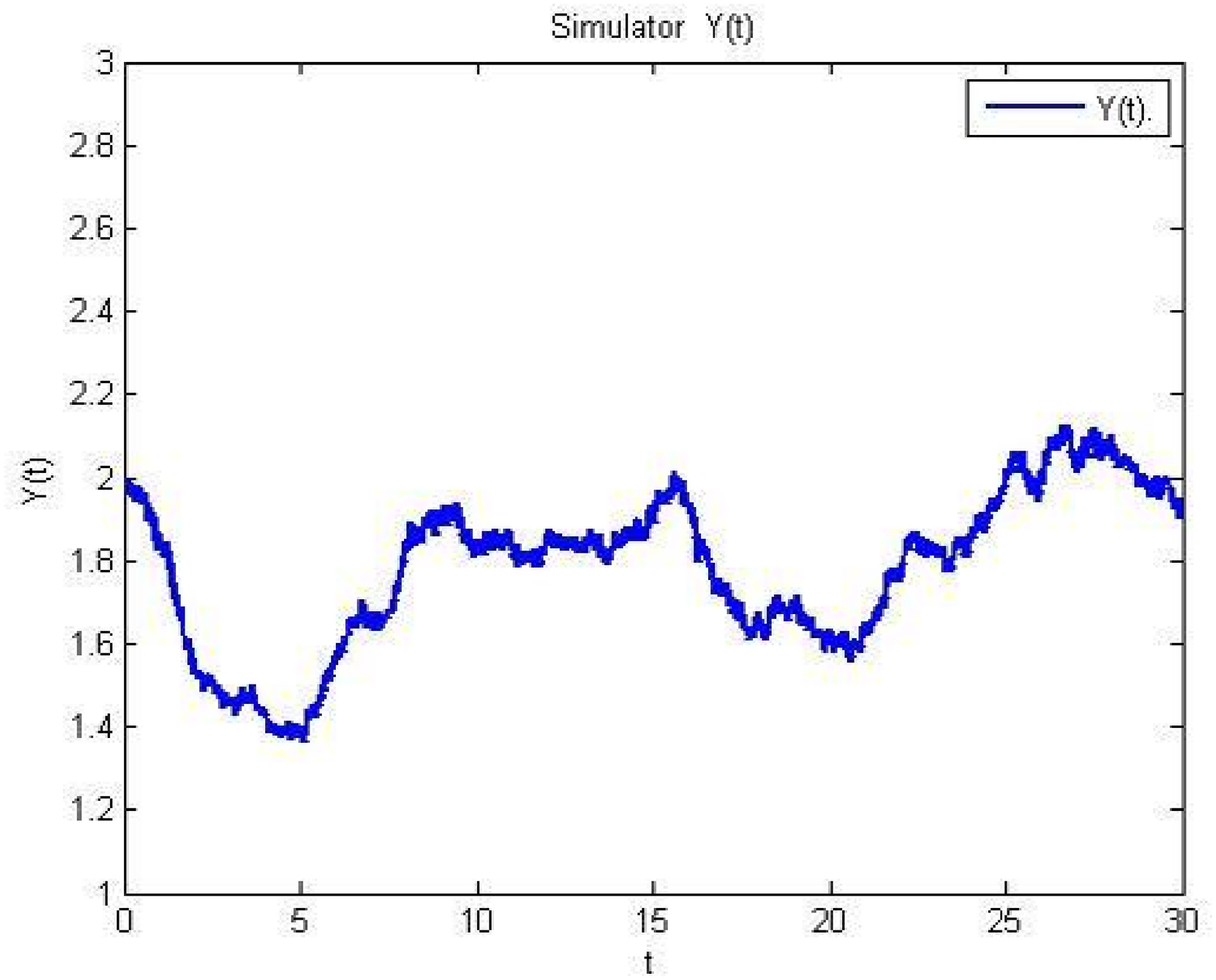}
\includegraphics[totalheight=2.2in,width=2.1in]{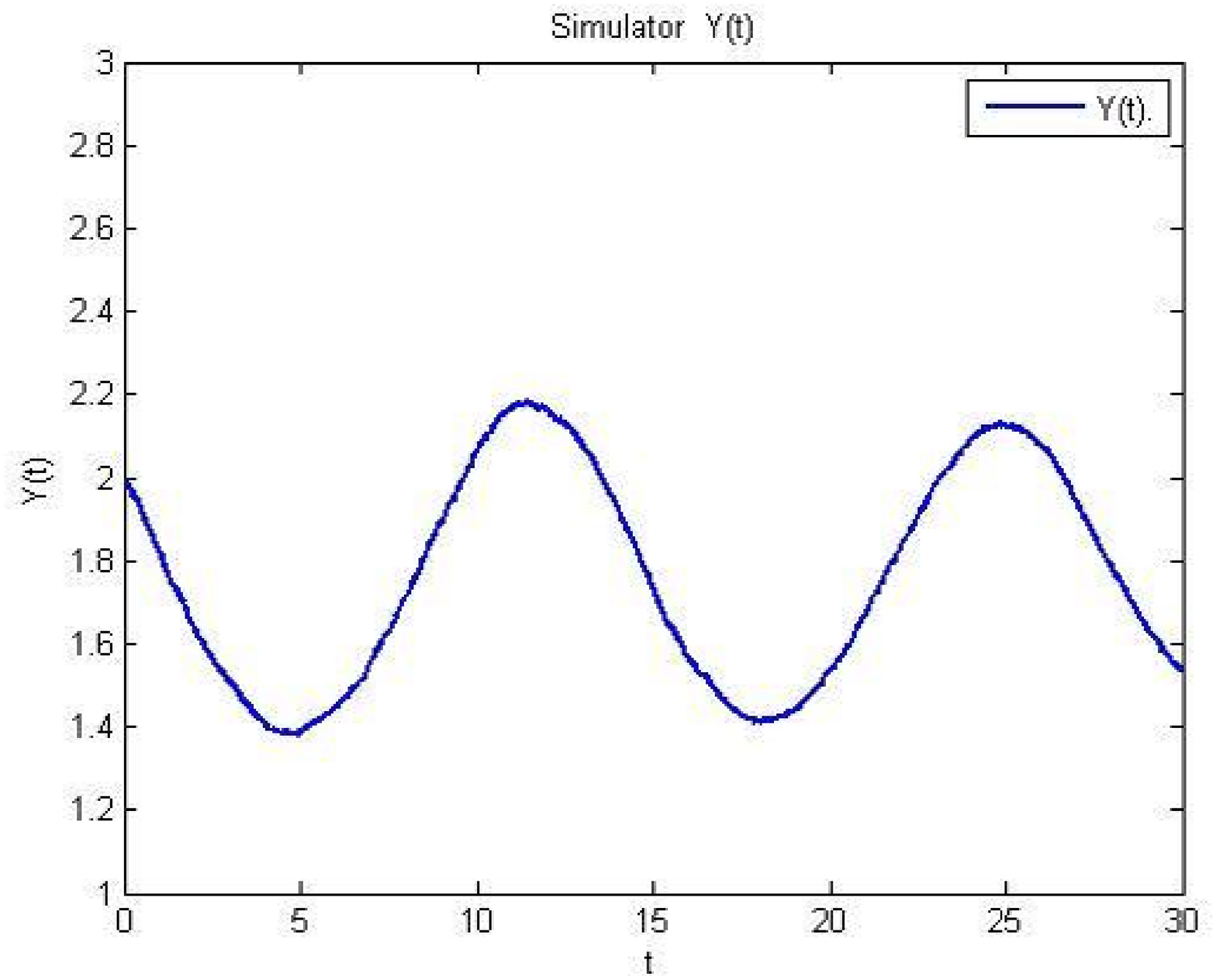}
\includegraphics[totalheight=2.2in,width=2.1in]{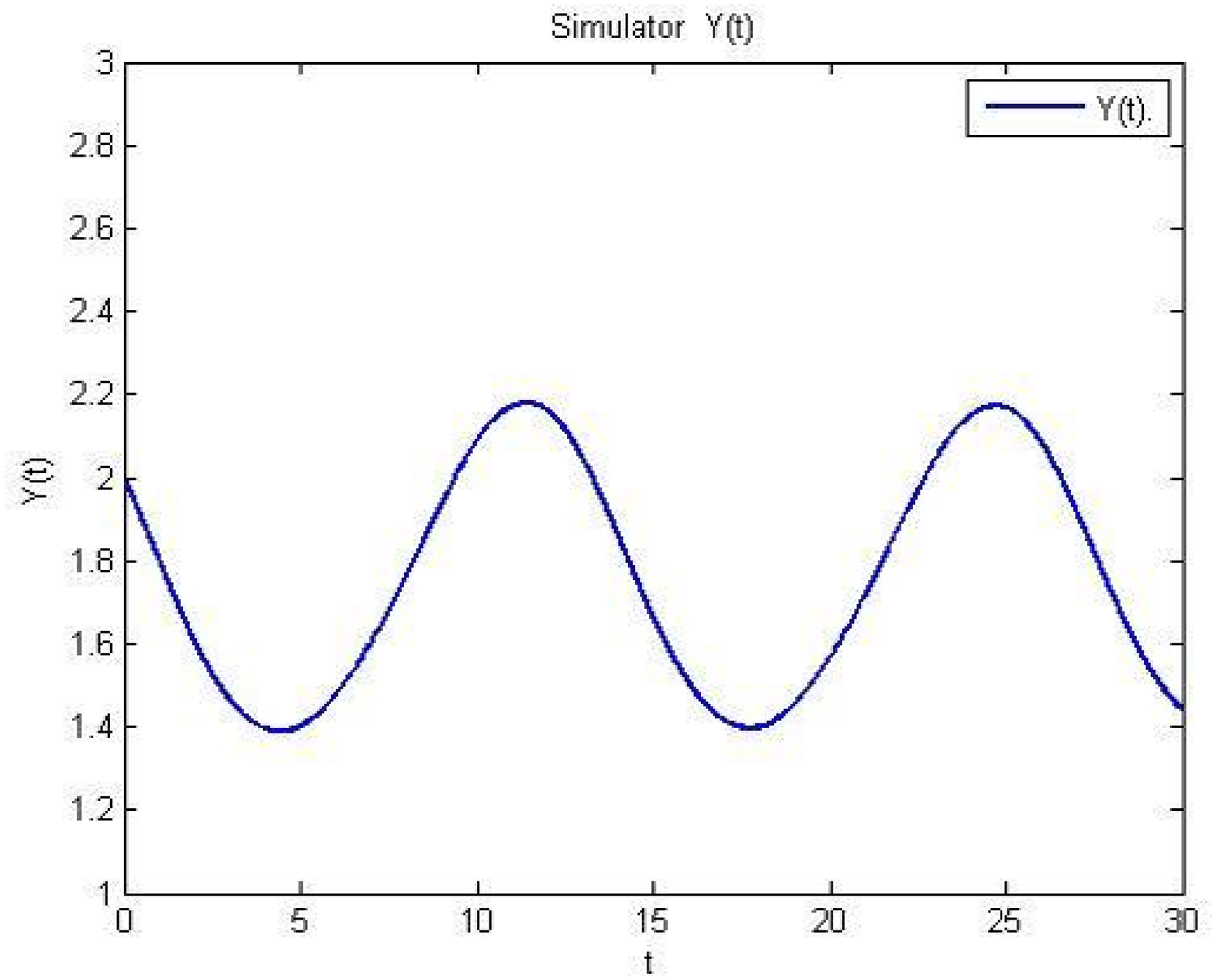}
\caption{From left to right: Graphs of $y^{\eps,\delta}(t)$ of solutions to Eq. \eqref{nex1} with $(\eps,\delta)=(0.01, 0.01)$, $(\eps,\delta)=(0.001, 0.001)$ and $y(t)$ to the averaged Eq. \eqref{nex2} respectively.}
\end{figure}

\begin{figure}[h]
\centering
\includegraphics[totalheight=2.2in,width=2.1in]{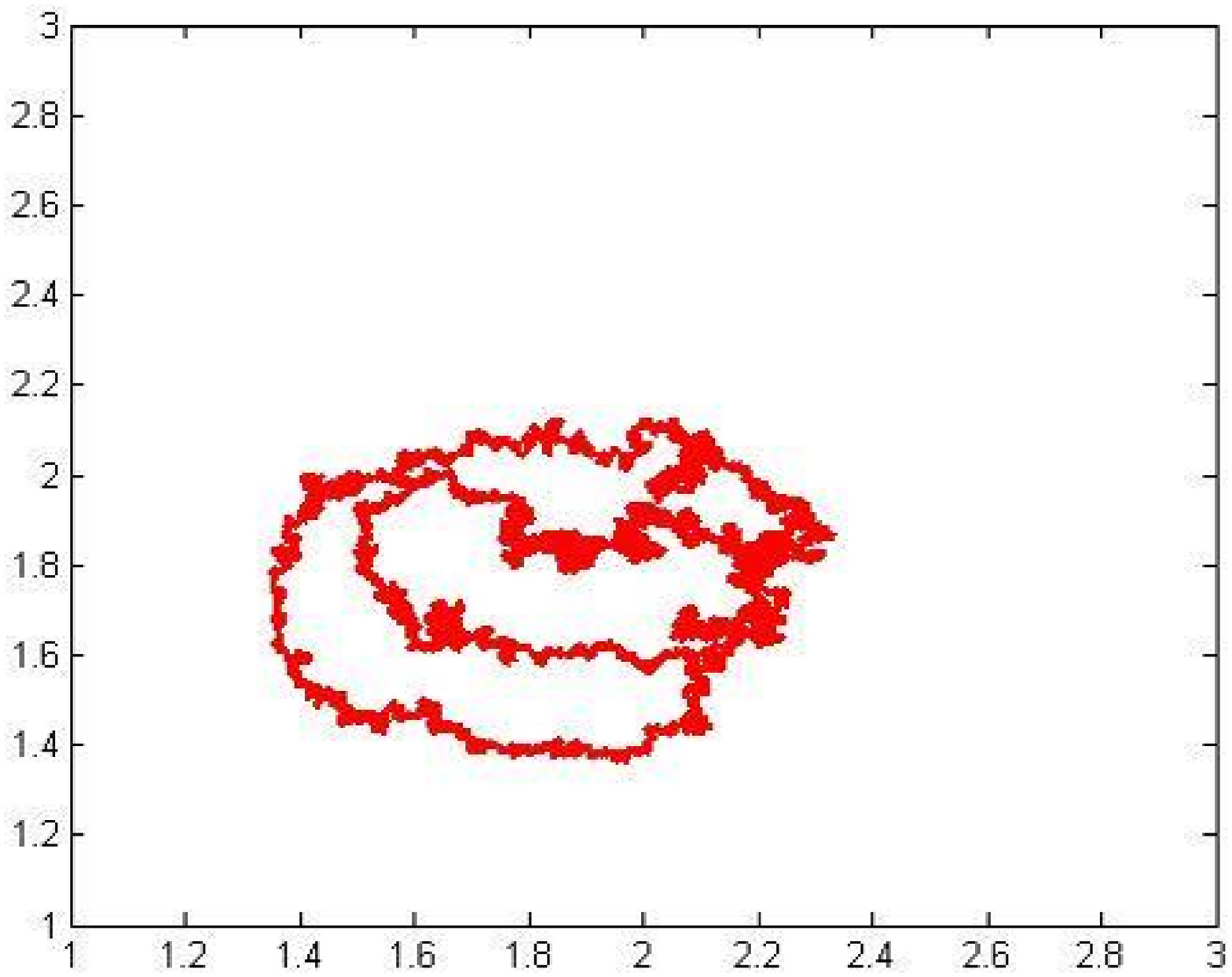}
\includegraphics[totalheight=2.2in,width=2.1in]{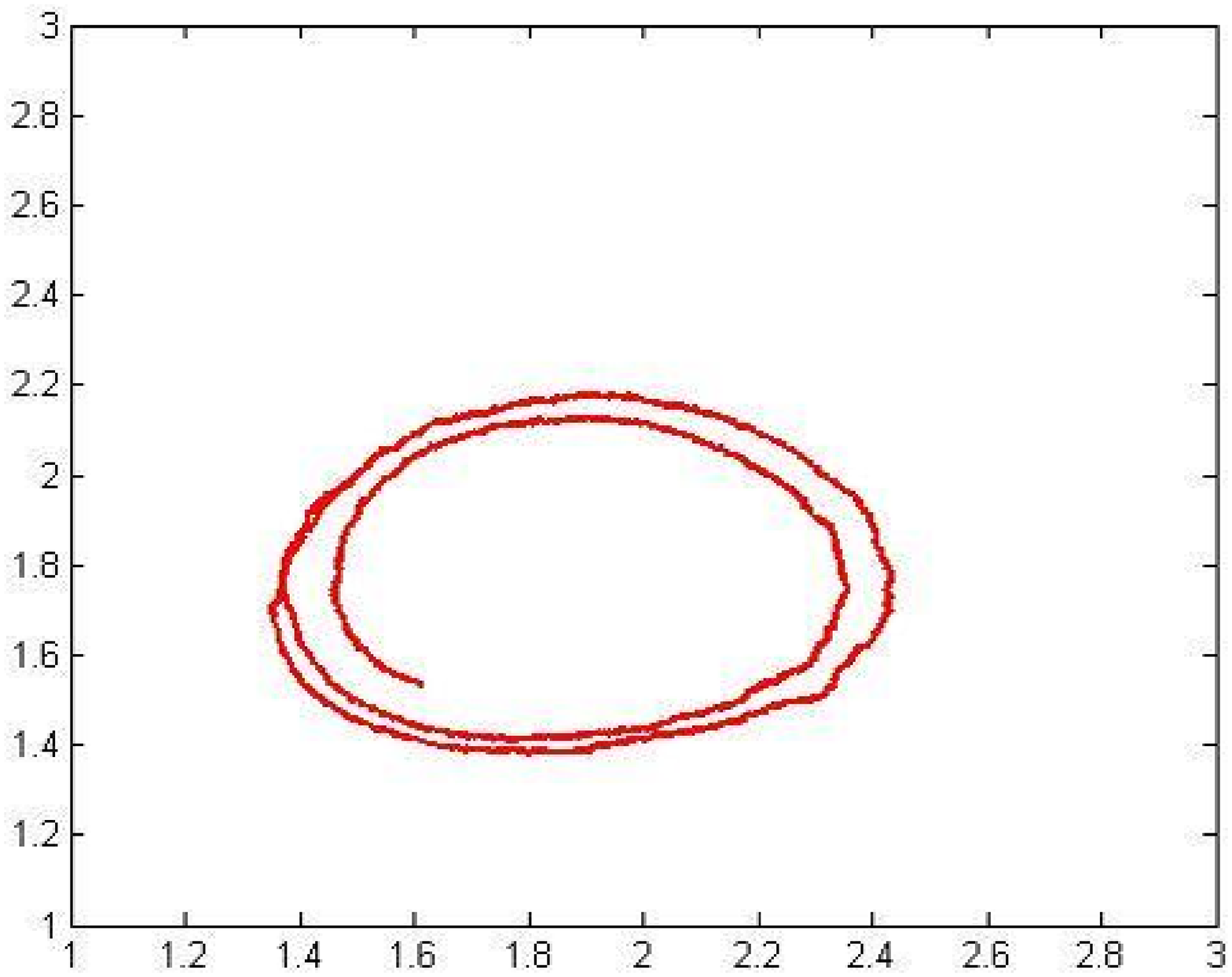}
\includegraphics[totalheight=2.2in,width=2.1in]{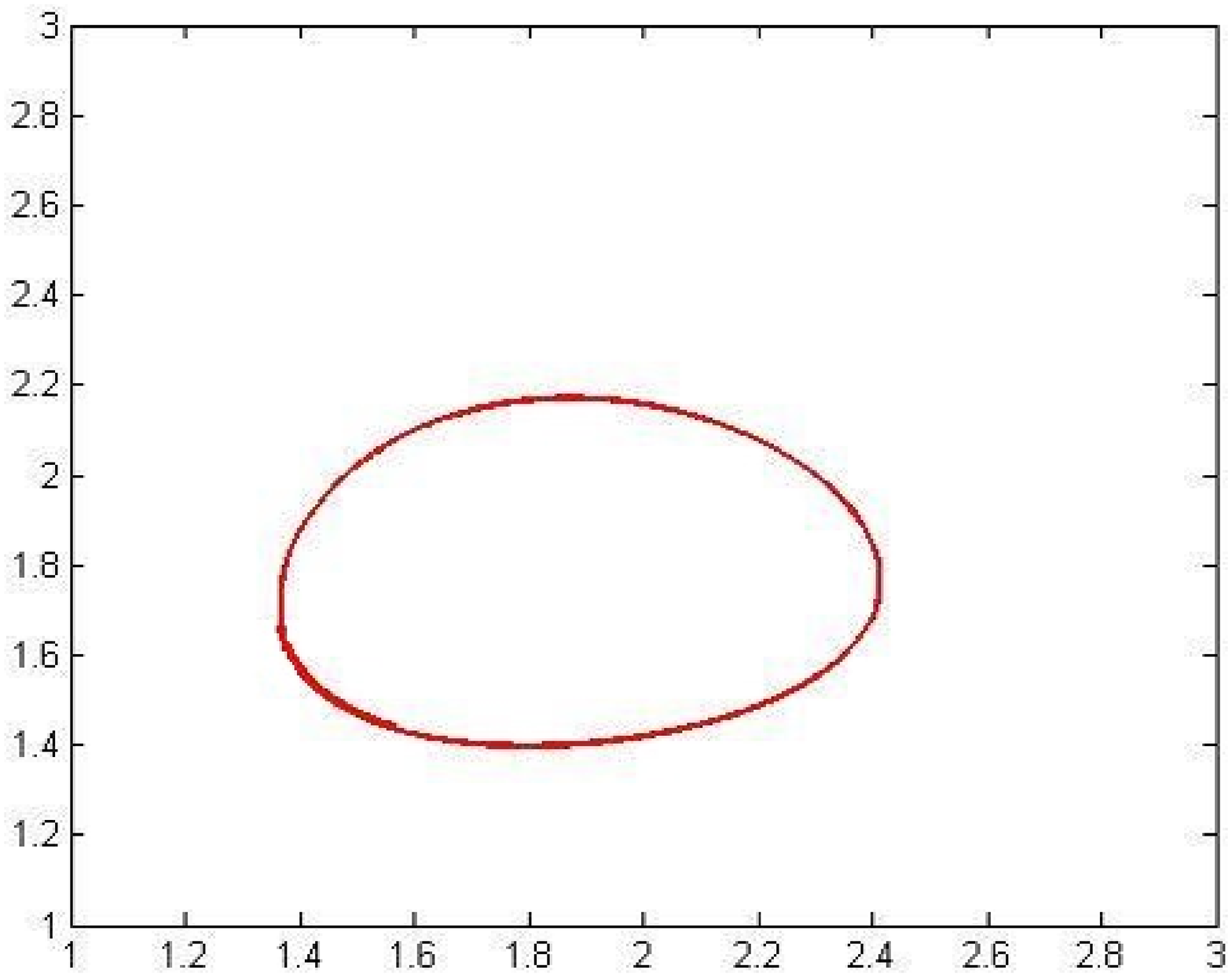}
\caption{From left to right: Phase portraits of sample solutions to Eq. \eqref{nex1} with $(\eps,\delta)=(0.01, 0.01)$, $(\eps,\delta)=(0.001, 0.001)$ and to the averaged Eq. \eqref{nex2} respectively.}
\end{figure}

\end{document}